\documentclass[11pt,reqno]{amsart}
\usepackage{amsfonts}
\usepackage{amssymb}
\usepackage{stmaryrd}
\usepackage{color}
\hoffset=- 2cm \voffset=- 0cm
\theoremstyle{plain}
\newtheorem{Thm}{Theorem}
\newtheorem{Assump}[Thm]{Assumption}

\newtheorem{Prop}[Thm]{Proposition}
\newtheorem{Rem}[Thm]{Remark}
\newtheorem{Lem}[Thm]{Lemma}
\newtheorem{Cor}[Thm]{Corollary}

\newtheorem{Def}[Thm]{Definition}

\newcommand {\p}{\partial}
\newcommand{\q}{\quad}
\newcommand{\qq}{\qquad}

\newcommand{\eq}{\begin{equation}}
\newcommand{\eeq}{\end{equation}}

\def\a{\alpha}

\def\g{\gamma}
\def\k{\kappa}
\def\lam{\lambda}

\def\p{\partial}

\def\O{\Omega}

\def\s{\sigma}
\def\var{\varepsilon}

\def\lam{\lambda}
\def\var{\varepsilon}

\def\A{\bold A}
\def\B{\bold B}
\def\D{\bold D}
\def\E{\bold E}
\def\e{\bold e}

\def\f{\bold f}

\def\H{\bold H}
\def\h{\bold h}
\def\j{\bold j}
\def\J{\bold J}

\def\u{\bold u}
\def\U{\bold U}
\def\bv{\bold v}

\def\w{\bold w}

\def\y{\bold y}
\def\Y{\bold Y}
\def\z{\bold z}
\def\Z{\bold Z}
\def\0{\bold 0}

\def\mA{\mathcal A}
\def\mB{\mathcal B}

\def\mE{\mathcal E}

\def\mH{\mathcal H}

\def\mL{\mathcal L}

\def\mP{\mathcal P}

\def\mT{\mathcal T}
\def\mU{\mathcal U}

\def\q{\quad}
\def\qq{\qquad}
\def\qqq{\qq\qq}

\def\curl{\text{\rm curl\,}}
\def\div{\text{\rm div\,}}

\def\loc{\text{\rm loc}}

\def\grad{\text{\rm grad\,}}

\errorcontextlines=0 \numberwithin{equation}{section}
\numberwithin{Thm}{section}
\begin{document}
\large
\title[Magneto]{General Magneto-Static Model}

\author[X.B.Pan]{Xing-Bin Pan }

\address{School of Mathematics,
East China Normal University,  and NYU-ECNU Institute of Mathematical Sciences at NYU Shanghai,
Shanghai 200062, P.R. China. xbpan@math.ecnu.edu.cn}




\begin{abstract}
In this paper we study a nonlinear magneto-static model on a general domain which is multiply-connected and has $m$ holes, and under a nonlinear relation between magnetic induction $\B$ and magnetic field $\H$. The equation contains a Neumann field $\h_1\in\Bbb H_1(\O)$ and a Dirichlet field $\h_2\in\Bbb H_2(\O)$, which represent the effects of domain topology. For a general electric current, the equation contains an unknown gradient $\nabla p(x)$, which represents the electric field. Existence results of solutions of boundary value problems of this model under various types of boundary conditions are proved, which exhibit the effects of domain topology.
\end{abstract}

\maketitle


\section{Introduction}

In this paper we study a nonlinear magneto-static model on a general domain in $\Bbb R^3$, which is multiply-connected and has $m$ holes (namely the boundary of $\p\O$ consists of $m+1$ connected components). We are interested in the effects of domain topology on solvability of the boundary value problems.

In the recent years quasilinear magneto-static models in electromagnetism have been studied by many authors. A typical equation is the quasilinear Maxwell type equation of the following form:
\eq\label{1}
\curl [\mH(x,\curl\u)]=\J(x),\qq \div\u=0\q \text{in }\O,\\
\eeq
where $\J$ is a given vector field. This system is a time-independent version of the eddy current model for a magnetic and anisotropic material where the relation between magnetic induction field $\B$ and the magnetic field $\H$ is given by the BH-curve $\H=\mH(\B)$ which is usually nonlinear,
see For instance, Milani and Picard  \cite{MP1, MP2}, R. Picard \cite{Pi}, Jochmann \cite{Jo}, F. Bachinger, U. Langer and J. Sch\"oberl \cite{BLS}, Jiang and Zheng \cite{JZ1, JZ2}, L. Yousept \cite{Yo},  Pan \cite{P4}, and the references therein.

It was shown in \cite{P4} that for a system similar to \eqref{1} with $\J$ replaced by a nonlinear vector-valued function $\f(x,\u)$, it is in general necessary to introduce a potential term into the equation and one needs to consider a Maxwell-Stokes type system
\eq\label{2}
\curl [\mH(x,\curl\u)]=\f(x,\u)+\nabla p,\qq \div\u=0\q \text{in }\O.
\eeq
It was also shown in \cite{P4} that the type of boundary condition for $p$ should be determined according to the topology of $\O$. If $\O$ has no holes, then it is natural to consider Dirichlet boundary condition for $p$. If $\O$ has holes, then one has to consider Neumann type boundary condition for $p$. In this paper we will explain the physical significance of the term $\nabla p$.

In this paper we assume that the relation between magnetic induction $\B$ and magnetic field  intensity $\H$ is nonlinear:
$$\B=\mB(x,\H),\q \text{or}\q
\H=\mH(x,\z).
$$
To avoid the complexity of electric part, we assume that the relation between electric displacement $\D$ and electric field intensity $\E$ is linear: $\D=\var\E,$ where $\var$ is scalar or matric valued function. Finally, we shall assume the electric current density $\j$ satisfies the Ohm's law. We shall introduce a vector field $\u$ and a scalar function $p$, such that the electric field $\E$ and magnetic induction $\B$ can be represented by
\eq\label{EB-up}
\s\E=\nabla p+\h_1,\q \B=\curl\u+\h_2,
\eeq
where $\h_1\in\Bbb H_1(\O)$, $\h_2\in\Bbb H_2(\O)$, which describe the topology of the domain $\O$. Then we can derive the magneto-static model. If the electric current is given, then the equation is
\eq\label{M6}
\left\{\aligned
&\curl [\mH(x,\curl\u+\h_2(x))]=\J(x)+\h_1(x)\q &\text{in }\O,\\
&\div\u=0\q &\text{in }\O;
\endaligned\right.
\eeq
and if the electric current depends on the magnetic induction, then the equation is
\eq\label{M7}
\left\{\aligned
&\curl [\mH(x,\curl\u+\h_2(x))]=\f(x,\curl\u+\h_2)+\h_1(x)+\nabla p\q &\text{in }\O,\\
&\div\u=0\q &\text{in }\O.
\endaligned\right.
\eeq
We shall examine solvability of \eqref{M6} subjected to a boundary condition on $\u$. For \eqref{M7} we need also a boundary condition for $p$. Following the observation in \cite{P4}, we shall consider Neumann type boundary condition for $p$. Regularity of weak solutions of BVPs of \eqref{M6} and \eqref{M7} can be obtained by the methods in \cite{P4} and will not be discussed in this paper.

We mention that the gradient term $\nabla p$ in \eqref{M7} has a clear physical meaning. In fact, from \eqref{EB-up}  we see that
$$
\nabla p=\s\E-\h_1,
$$
so $\nabla p$ represents both electric field $\E$ and domain topology.
We also mention that both \eqref{M6} and \eqref{M7} contain the topological terms $\h_1, \h_2$, and we will see that solvability of the boundary value problems (BVPs for short) for \eqref{M6} and \eqref{M7} depends on the choice of $\h_1, \h_2$, which is an effect of domain topology on these equations.

In this paper we shall use various results on vector fields, including the divergence-curl-gradient inequalities, and regularity of div-curl system and of time-independent Maxwell's system, which can be fund in literature, see for instance \cite{DaL, GR, MMT, Sc, W, BW, NW, AS, AuA, CDN, KY} and the references therein.
 We mention that time-independent semilinear Maxwell system has also been studied by many authors, see for instance \cite{Y, BF, JS, P3} and the references therein. We also mention that the Meissner states of type I\!I superconductors can be described by a quasilinear system of the form similar to \eqref{1}, see \cite{Ch, Mon, BaP, P1, LP}.

This paper is organized as follows. In section 2 we list notation for various spaces of vector fields, and collect some facts on these spaces that we will need. In section 3 we derive the magneto-static equations. Sections 4, 5, and 6 are devoted to  study of existence of weak solutions to BVPs of \eqref{M6} under various type of boundary conditions on $\u$. Solvability of BVPs for \eqref{M7} is studied in section 7.

\section{Preliminaries}

\subsection{Notation}\

Most of materials in this subsection can be fund in \cite{DaL, GR}.
 Let $\O$ be a bounded domain in $\Bbb R^3$ and $\nu$ be the unit outer normal vector of the boundary $\p\O$.
We assume that $\O$ has the following properties (\cite[p.217]{DaL}):
\begin{itemize}
\item[$(O_1)$] $\O$ is a bounded domain in $\Bbb R^3$ with a $C^r$ boundary
$\p\O$, $r\geq 1$; $\O$ is locally situated on one side of
$\p\O$; $\p\O$ has a finite number of connected components
$\Gamma_1, \cdots, \Gamma_{m+1}$, where $m\geq 0$ and
$\Gamma_{m+1}$ denoting the boundary of the unbounded connected
component of the set $\Bbb R^3\setminus\bar\O$.
\item[$(O_2)$] There exist $N$ manifolds of dimension $2$ and of class $C^r$
denoted by $\Sigma_1,\cdots,\Sigma_N$, $N\geq 0$, such that
$\Sigma_i\cap \Sigma_j=\emptyset$ for $i\neq j$ and non-tangential
to $\p\O$, such that $\dot\O=\O\setminus(\sum_{j=1}^N\Sigma_j)$ is simply-connected and Lipschitz.
\end{itemize}

\subsubsection{The spaces of vector fields}\

As in \cite{BaP, P1, P4}, we denote the Sobolev spaces of scalar functions by $W^{k,p}(\O)$, $H^k(\O)$ and $H^s(\p\O)$ etc, and denote the Sobolev spaces of vector fields by $W^{k,p}(\O,\Bbb R^3)$, $H^k(\O,\Bbb R^3)$, $H^s(\p\O,\Bbb R^3)$ etc.
We use same notation for the norm of scalar functions and for vector fields. For instance, the norms in $H^k(\O)$ and in $H^k(\O,\Bbb R^3)$ are both denoted by $\|\cdot\|_{H^k(\O)}$.
We denote
$$\aligned
&\Bbb H_1(\O)=\{\u\in L^2(\O,\Bbb R^3):~ \curl\u=\0\;\;\text{and}\;\; \div\u=0\;\;\text{in }\O,\; \nu\cdot\u=0\;\;\text{on }\p\O\},\\
&\Bbb H_2(\O)=\{\u\in L^2(\O,\Bbb R^3):~ \curl\u=\0\;\;\text{and}\;\; \div\u=0\;\;\text{in }\O,\; \u_T=\0\;\;\text{on }\p\O\}.
\endaligned
$$
It is well-known that $\dim\Bbb H_1(\O)=N$ and $\dim\Bbb H_2(\O)=m$, where $m$ and $N$ are given in $(O_1)$ and $(O_2)$. We denote by $\Bbb H_j(\O)^\perp_{L^2(\O)}$ the orthogonal complementary of $\Bbb H_j(\O)$ in $L^2(\O,\Bbb R^3)$. When $\p\O$ is of $C^2$, using regularity of div-curl system we have
\eq\label{H1H2}
\|\h_j\|_{H^1(\O)}\leq C_j(\O)\|\h_j\|_{L^2(\O)},\q \forall \h_j\in\Bbb H_j(\O),\; j=1,2.
\eeq

If $X(\O)$ denotes a space of vector fields on $\O$, we write
$$\aligned
&X(\O,\div0)=\{\u\in X(\O): \div\u=0\;\;\text{in }\O\},\\
&X(\O,\curl0)=\{\u\in X(\O): \curl\u=\0\;\;\text{in }\O\},\\
&X_{t0}(\O)=\{\u\in X(\O): \u_T=\0\;\;\text{on }\p\O\},\\
&X_{n0}(\O)=\{\u\in X(\O): \nu\cdot\u=0\;\;\text{on }\p\O\}.
\endaligned
$$
For instance,
$$\aligned
&\mH(\O,\curl,\div)=\{\w\in L^2(\O,\Bbb R^3): \curl\w\in L^2(\O,\Bbb R^3),\; \div\w\in L^2(\O)\},\\
&\mH(\O,\curl,\div0)=\{\w\in \mH(\O,\curl,\div),\; \div\w=0\},\\
&\mH(\O,\curl0,\div0)=\{\w\in \mH(\O,\curl,\div): \curl\w=\0,\; \div\w=0\}.
\endaligned
$$
If $Y(\O)$ is a space of scalar functions on $\O$, we denote
$$\dot Y(\O)=\{\phi\in X(\O): \int_\O \phi dx=0\}.
$$

Denote
$$\aligned
&L^{2,-1/2}_t(\O,\Bbb R^3)=\{\w\in L^2(\O,\Bbb R^3): \nu\times\w\in H^{-1/2}(\p\O,\Bbb R^3)\},\\
&L^{2,-1/2}_\nu(\O,\Bbb R^3)=\{\w\in L^2(\O,\Bbb R^3): \nu\cdot\w\in H^{-1/2}(\p\O)\},\\
&L^{2,-1/2}_{t0}(\O,\Bbb R^3)=\{\w\in L_t^{2,-1/2}(\O,\Bbb R^3): \nu\times\w=\0\;\text{on }\p\O\},\\
&L^{2,-1/2}_{\nu0}(\O,\Bbb R^3)=\{\w\in L^{2,-1/2}_\nu(\O,\Bbb R^3): \nu\cdot\w=0\;\text{on }\p\O\}.
\endaligned
$$
We also need spaces of tangential vector fields:
$$\aligned
T\!H^{k+1/2}(\p\O,\Bbb R^3)=&\{\u\in H^{k+1/2}(\p\O,\Bbb R^3):~ \nu\cdot\u=0\},\\
T\!C^{k,\a}(\p\O,\Bbb R^3)=&\{\u\in C^{k,\alpha}(\p\O,\Bbb R^3):~ \nu\cdot\u=0\}.
\endaligned
$$
We denote
\eq\label{1/2}
\aligned
&\langle\phi,\eta\rangle_{\p\O,1/2}=\langle\phi,\eta\rangle_{H^{-1/2}(\p\O),H^{1/2}(\p\O)},\\
&\dot H^{-1/2}(\p\O)=\{\zeta\in H^{-1/2}(\p\O)~:~ \langle  \zeta, 1\rangle_{\p\O,1/2}=0\}.
\endaligned
\eeq

\subsubsection{Decomposition of $L^2(\O,\Bbb R^3)$}\

Let us recall that
(see \cite[p.225-226]{DaL}):
\eq\label{L2-decom}
\aligned
&L^2(\O,\Bbb R^3)=\mathcal H(\O,\curl0)\oplus_{L^2(\O)} \curl H^1_{t0}(\O,\div0),\\
&L^2(\O,\Bbb R^n)=\mathcal H(\O,\div0)\oplus_{L^2(\O)} \grad H^1_0(\O),
\endaligned
\eeq
and recall image of operator $\curl$  (see \cite[p.222, Proposition 3; p.226, Remark 5]{DaL}):
\eq\label{im-curl}
\aligned
\curl H^1(\O,\Bbb R^3)=\curl H^1_{n0}(\O,\div0)=&\mathcal H^\Gamma(\O,\div0),\\
\curl H^1_{t0}(\O,\div0)=&\mathcal H_0^\Sigma(\O,\div0),
\endaligned
\eeq
where
$$\aligned
\mathcal H^\Gamma(\O,\div0)=&\{\u\in \mathcal H(\O,\div0): \langle \u\cdot\nu, 1\rangle_{H^{-1/2}(\Gamma_j),H^{1/2}(\Gamma_j)}=0,\; j=1,\cdots, m+1\},\\
\mathcal H_0^\Sigma(\O,\div0)
=&\{\u\in L^2(\O,\Bbb R^3): \div\u=0,\; \u\cdot\nu|_{\p\O}=0,\\
&\qqq \langle \u\cdot\nu, 1\rangle_{H^{-1/2}(\Sigma_i),H^{1/2}(\Sigma_i)}=0,\; i=1,\cdots, N\}.
\endaligned
$$
These yield the following decompositions of the kernel of the operators $\curl$ and $\div$:
\eq\label{dec-0}
\aligned
\mathcal H(\O,\curl0)=&\grad H^1(\O)\oplus_{L^2(\O)}\Bbb H_1(\O),\\
\mathcal H(\O,\div0)=& \mH^\Gamma(\O,\div0)\oplus_{L^2(\O)}\Bbb H_2(\O),\\
\mathcal H_0(\O,\div0)=&\mathcal H^\Sigma_0(\O,\div0)\oplus_{L^2(\O)}\Bbb H_1(\O).
\endaligned
\eeq
It follows that $\Bbb H_1(\O)\subset \mH^\Gamma(\O,\div0)$ and $\Bbb H_2(\O)\perp_{L^2(\O)}\Bbb H_1(\O)$.

From the first equality of \eqref{L2-decom} and the first equality of \eqref{dec-0} we have
\eq\label{dec-1}
L^2(\O,\Bbb R^3)=\mH^\Sigma_0(\O,\div0)\oplus_{L^2(\O)} \Bbb H_1(\O)\oplus_{L^2(\O)} \text{grad} H^1(\O).
\eeq
From the second equality of \eqref{L2-decom} and the second equality of \eqref{dec-0} we have
\eq\label{dec-2}
L^2(\O,\Bbb R^3)=\mH^\Gamma(\O,\div0)\oplus_{L^2(\O)} \Bbb H_2(\O)\oplus_{L^2(\O)} \text{grad} H^1_0(\O).
\eeq

\subsection{Assumptions on $\mH(x,\z)$}\

When $\O$ is bounded, we say $\mH\in C^{k,\a}_{\loc}(\bar\O\times\Bbb R^3,\Bbb R^3)$ if $\mH\in C^{k,\a}_{\loc}(\bar\O\times K,\Bbb R^3)$ for any bounded set $K\subset \Bbb R^3$.
Now we list the conditions on $\mH(x,\z)$ that we may need.
\begin{itemize}
\item[$(H_1)$] $\mH\in C^0_{\loc}(\bar\O\times \Bbb R^3,\Bbb R^3)$, and there exist positive constants $c_1, c_2$, and functions $g_1\in L^1(\O)$, $g_2\in L^2(\O)$ such that, for all $x\in\bar\O$ and $\z\in\Bbb R^3$,
$$
c_1|\z|^2-g_1(x)\leq \langle \mH(x,\z),\z\rangle,\qq  |\mH(x,\z)|\leq c_2|\z|+g_2(x).
$$
\item[$(H_2)$] $\mH\in C^1_{\loc}(\bar\O\times \Bbb R^3,\Bbb R^3)$, and there exist positive constants $\mu, M_1, M_2$ such that for all $x\in\bar\O$ and $\z,\; \xi\in\Bbb R^3$,
$$
\aligned
& \langle \nabla_\z \mH(x,\z)\xi,\xi\rangle\geq \mu|\xi|^2,\\
&|\nabla_x \mH(x,\z)|\leq M_1(|\z|+1),\qq
 |\nabla_\z\mH(x,z)|\leq M_2.
\endaligned
$$
\item[$(H_3)$] $\mH(x,\z)$ has an inverse $\mB(x,\w)$, namely,
$\z=\mB(x,\w)$ if and only if $\w=\mH(x,\z)$.
\end{itemize}

\begin{Rem} If $\mH$ satisfies $(H_2)$ then it is Lipschitz in $\z$, and has the following property:
\begin{itemize}
\item[$(H_4)$] There exists a constant $\mu>0$ such that
$$\langle \mH(x,\z_2)-\mH(x,\z_1),\z_2-\z_1\rangle\geq\mu|\z_2-\z_1|^2.
$$
\end{itemize}
\end{Rem}

We can show that $(H_1), (H_2), (H_3)$ imply the following properties of $\mB$:

\begin{itemize}
\item[$(B_1)$] $\mB\in C^0_{\loc}(\bar\O\times \Bbb R^3,\Bbb R^3)$, and there exist positive constants $c_3, c_4$, and functions $g_3\in L^1(\O)$, $g_4\in L^2(\O)$ such that, for all $x\in\bar\O$ and $\w\in\Bbb R^3$,
$$
c_3|\w|^2-g_3(x)\leq \langle \mB(x,\w),\w\rangle, \q |\mB(x,\w)|\leq c_4|\w|+g_4(x).
$$
\item[$(B_2)$] $\mB\in C^1_{\loc}(\bar\O\times\Bbb R^3)$,  and there exist positive constants $\lam_0, N_1, N_2$ such that for all $x\in\bar\O$ and $\w,\; \eta\in\Bbb R^3$,
$$
\aligned
&\langle(\nabla_\w\mB(x,\w))\eta,\eta\rangle\geq\lam_0|\eta|^2,\\
&|\nabla_x\mB(x,\w)|\leq N_1(|\w|+1),\qq
|\nabla_\w\mB(x,\w)|\leq N_2.
\endaligned
$$
\end{itemize}

\begin{Rem}\label{Rem4.2}
\begin{itemize}
\item[(i)] If $\H$ satisfies $(H_1)$, $(H_3)$, then $\mB$ satisfies $(B_1)$, with
$$c_3={c_1\over 2c_2^2},\q c_4={1\over c_1},\q g_3(x)=g_1(x)+c_3g_2(x)^2,\q  g_4(x)=\sqrt{2|g_1(x)|\over c_1}.
$$
\item[(ii)] If $\H$ satisfies $(H_2)$, $(H_3)$, then $\mB$ satisfies $(B_2)$, where $\lam_0$ and $N_2$ depend on $\mu$ and $M_2$, $N_1$ depends on $\mu, M_1, M_2$. In fact,
$$ \lam_0={\mu\over M_2^2},\q      N_2= {9C_0M_1M_2^2\over \mu^3},\q          N_2= {C_0M^2_2\over \mu^3},
$$
where $C_0$ is an absolute constant.
\end{itemize}
\end{Rem}

\begin{Rem}\label{Lem-mono-B} If $(B_2)$ holds, then for any $\w_1,\w_2\in\Bbb T^3$ and $x\in\bar\O$,
\eq\label{mono-B}
\langle \mB(x,\w_2)-\mB(x,\w_1), \w_2-\w_1\rangle\geq \lam_0|\w_2-\w_1|^2.
\eeq
\end{Rem}

\subsection{Assumption on $\f(x,\z)$}\

\begin{itemize}
\item[$(f_1)$]  $\f\in C^1_{\loc}(\bar\O\times\Bbb R^3,\Bbb R^3)$, and there exist positive constants $K_0, K_1, K_2$ and $f_0\in L^2(\O)$ such that, for all $x\in\bar\O$ and $\z\in \Bbb R^3$,
$$\aligned
&|\f(x,\z)|\leq K_0|\z|+f_0(x),\q |\nabla_{x}\f(x,\z)|\leq K_1(|\z|+1),\\
&|\nabla_{\z}\f(x,\z)|\leq K_2.
\endaligned
$$
\end{itemize}

\subsection{The Projections $\mP_\nu$ and $\mP_n$}\label{SecA}

\begin{Def}
We introduce the projection $\mP_\nu$ as follows. If $\w\in L^2(\O,\Bbb R^3)$, then $\mP_\nu[\w]=\w+\nabla\phi_\w$, where $\phi_\w\in \dot H^1(\O)$ is the unique weak solution of
\eq\label{Neumann}
-\Delta\phi=\div\w\q\text{\rm in }\O,\qq
{\p \phi\over\p\nu}=-\nu\cdot\w\q\text{\rm on }\p\O,
\eeq
namely, for any $\phi\in H^1(\O)$ and
\eq\label{wkNeumann}
\int_\O(\nabla\phi+\w)\cdot\nabla\eta dx=0,\q\forall \eta\in H^1(\O).
\eeq
\end{Def}

\begin{Rem}\label{Lem-A.2} Assume $\O$ is  abounded domain in $\Bbb R^3$ with a $C^2$ boundary. Then
$\mP_\nu: L^2(\O,\Bbb R^3)\mapsto \mH^\Gamma(\O,\div0)$
is a continuous projection.
\end{Rem}

\begin{Def}
Assume $\nu\times\H^0$ satisfies \eqref{condMNa-J1}. We introduce the Neumann projection $\mP_n$ associated with $\nu\cdot\curl\H^0_T$ as follows. If $\w\in L^2(\O,\Bbb R^3)$, then $\mP_n[\w]=\w+\nabla\psi_\w$, where $\psi_\w\in \dot H^1(\O)$ is the unique weak solution of
\eq\label{Neumann-curl}
-\Delta\psi=\div\w\q\text{\rm in }\O,\qq
{\p \psi\over\p\nu}=\nu\cdot\curl\H^0_T-\nu\cdot\w\q\text{\rm on }\p\O,
\eeq
namely, for all $\eta\in H^1(\O)$ it holds that
$$
\int_\O(\nabla\psi+\w)\cdot\nabla\eta dx=\langle\nu\cdot\curl\H^0_T,\eta\rangle_{\p\O, 1/2}.
$$
\end{Def}

\begin{Lem}\label{Lem-A.4} Assume $\O$ is a bounded domain in $\Bbb R^3$ with a $C^2$ boundary, and $\nu\times\H^0$ satisfies \eqref{condMNa-J1}. Then
$\mP_n: L^2(\O,\Bbb R^3)\to \mH^\Gamma(\O,\div0)$
is a continuous projection.
\end{Lem}
\begin{proof} Denote by $\H^0$ the divergence-free and tangential component preserving extension on $\H^0_T$. For any $\w\in L^2(\O,\Bbb R^3)$, $
\mP_n[\w]=\w+\nabla\psi_\w\in \mH(\O,\div0)$. Hence $\nu\cdot\mP_n[\w]\in H^{-1/2}(\p\O)$. Denote the connected components of $\p\O$ by $\Gamma_j$, $j=1,\cdots, m+1$, and let $\eta_j$ denote the harmonic function in $\O$ such that $\eta_j=\delta_{ij}$ on $\Gamma_i$. Then $\nu\times\nabla\eta_j=\0$ on $\p\O$. So
$$
\aligned
&\langle \nu\cdot\mP_n[\w],1\rangle_{H^{-1/2}(\Gamma_j),H^{1/2}(\Gamma_j)}
= \langle \nu\cdot\curl\H^0_T,\eta_j\rangle_{\p\O,1/2}
=\int_\O\nabla\eta_j\cdot\curl\H^0 dx\\
=&\int_\O\div(\H^0\times\nabla\eta_j)dx
= -\langle \nu\times\nabla\eta_j,\H^0_T\rangle_{\p\O,1/2}=0.
\endaligned
$$
Hence $\mP_n[\w]\in \mH^\Gamma(\O,\div0)$. Taking  $\psi_\w$ as a test function we have
$$
\aligned
&\int_\O (\nabla\psi_\w+\w)\cdot\nabla\psi_\w dx =\langle \nu\cdot\curl\H^0_T,\psi_\w\rangle_{\p\O,1/2}\leq \|\nu\cdot\curl\H^0_T\|_{H^{-1/2}(\p\O)}\|\psi_\w\|_{H^{1/2}(\p\O)}.
\endaligned
$$
So
$$\aligned
&\|\nabla\psi_\w\|_{L^2(\O)}^2=\int_\O(\nabla\psi_\w+\w)\cdot \nabla\psi_\w dx-\int_\O\w\cdot\nabla\psi_\w dx\\
\leq &\|\nu\cdot\curl\H^0_T\|_{H^{-1/2}(\p\O)}\|\psi_\w\|_{H^{1/2}(\p\O)}+\|\w\|_{L^2(\O)}\|\nabla\psi_\w\|_{L^2(\O)}.
\endaligned
$$
Using the trace theorem and applying the Poincar\'e inequality to $\psi\in \dot H^1(\O)$, we get
$$
\|\nabla\psi_\w\|_{L^2(\O)}\leq C_0(\O)\|\nu\cdot\curl\H^0_T\|_{H^{-1/2}(\p\O)}+\|\w\|_{L^2(\O)}.
$$
\end{proof}

\section{Derivation of the Magneto-static Equations}

\subsection{Time-dependent problem: quasi-static magnetic fields}\

\subsubsection{The Maxwell equations and our assumptions}\

We start with the general case  of magnetic and anisotropic materials, occupying a bounded $C^2$ domain in $\Bbb R^3$, which is multiply-connected  and has holes. The operators $\div$, $\curl$, $\nabla$ and $\Delta$ act on the space variables $x$.
The Maxwell equations are of the following form
\eq\label{Max}
\left\{\aligned
&\div(\var \E)=\rho,\\
&\div\B=0,\\
&-\p_t(\var\E)+\curl\H=\bold j,\\
&\p_t\B+\curl\E=\0,
\endaligned\right.
\eeq
where  $\E$ is the electric field, $\B$ is the magnetic induction, $\H$ is the magnetic field, $\bold j$ is the current density, $\rho$ is the charge density, and $\var$ is the permittivity.\footnote{From the first and third equations in \eqref{Max} we get
$$
\p_t\rho+\div\j=0,
$$
which is the conservation law of electric charge.}

\begin{Assump} (Assumption on the $\H$-$\B$ relation.) There exists a nonlinear vector-valued function $\mB(x,\z)$ such that
magnetic induction $\B$ can be represented by magnetic field $\H$:
\eq\label{eq-BH}
\B(t,x)=\mB(x,\H(t,x)).
\eeq
$\mB(x,\cdot)$ has an inverse $\mH(x,\cdot)$, which means that
\eq\label{H=HB}
\text{$\B=\mB(x,\H)$ holds if and only if $\H=\mH(x,\B)$}.
\eeq
\end{Assump}

\begin{Assump} (Assumption on the magnetic induction $\B$.)
The magnetic induction $\B(t,x)$ is such that
\eq\label{cond-1}
\B,\;\p_t\B \in L^2(0,T; L^2(\O,\Bbb R^3)),\q\forall T>0.
\eeq
\end{Assump}

\begin{Assump} (Assumption on the electric field $\E$.)
The electric field $\E(t,x)$ is such that
\eq\label{cond-3}
\E\in L^2(0,T; H^1(\O,\Bbb R^3)),\q\forall T>0.
\eeq
\end{Assump}

Later on we shall assume also the Ohm's law holds  (see for instance \cite[p.199-120]{LL}):
\begin{Assump} (Assumption on the current density $\j$.) The current density $\j$ satisfies the Ohm's law
\eq\label{Ohm}
\bold j=\j_a+\s\E,
\eeq
where $\s\geq 0$ is the electric conductivity, and $\bold j_a$ is the applied current density.
\end{Assump}

\subsubsection{Rewrite \eqref{Max} in term of $\E$ and a vector potential $\A$}\

{\it Step 1}. General case.
From \eqref{Max}(b) we have $\div\B=0$.
From the first line of \eqref{im-curl} and second line of \eqref{dec-0} we see that there exist vector fields $\w(t,x)$ and $\h_2(t,x)$, with $\w(t,\cdot)\in H^1_{n0}(\O,\div0)$ and $\h_2(t,\cdot)\in \Bbb H_2(\O)$, such that
\eq\label{decom-B}
\aligned
&\B(t,x)=\curl\w(t,x)+\h_2(t,x),\\
&\div\w(t,x)=0\q\text{in }\O,\q \nu\cdot\w(t,x)=0\q\text{on }\p\O.
\endaligned
\eeq
Since $\Bbb H_2(\O)$ is of finite dimensions, and $\curl \w(t,\cdot)$ is orthogonal to $\Bbb H_2(\O)$, it follows  that
\eq\label{cond-2}
\w,\;\p_t\w\in L^2(0,T; H^1(\O,\Bbb R^3)),\q \h_2,\;\p_t\h_2\in L^2(0,T; \Bbb H_2(\O)),\q\forall T>0.
\eeq

From \eqref{Max}(d) and \eqref{decom-B} we have
$$\aligned
\0=\p_t\B+\curl\E
=\p_t\curl\w+\p_t\h_2+\curl\E
=\curl(\E+\p_t\w)+\p_t\h_2.
\endaligned
$$
From this,  \eqref{cond-3} and \eqref{cond-2} we have, for any $t>0$,
$$-\p_t\h_2=\curl(\E+\p_t\w)\in \curl H^1(\O,\Bbb R^3)=\mH^\Gamma(\O,\div0)\subseteq [\Bbb H_2(\O)]^\perp_{L^2(\O)}.
$$
From this and the last equality in \eqref{cond-2} we have
$\p_t\h_2=\0$, so  $\curl(\E+\p_t\w)=-\p_t\h_2=\0$. By the first line of \eqref{dec-0}, there exists a scalar function $\xi(t,x)$ and a vector field $\tilde\h_1(t,x)$, with $\xi(t,\cdot)\in H^1(\O)$, $\tilde\h_1(t,\cdot)\in \Bbb H_1(\O)$, such that
\eq\label{EptW}
\E+\p_t\w=-\nabla\xi+\tilde\h_1.
\eeq
Using this, \eqref{cond-3} and \eqref{cond-2}, and since $\Bbb H_1(\O)$ is of finite dimensions and $\nabla\xi(t,\cdot)$ is orthogonal to $\Bbb H_1(\O)$, we can show that
$$
\nabla\xi\in L^2(0,T; H^1(\O,\Bbb R^3)),\q \tilde\h_1\in L^2(0,T; \Bbb H_1(\O)),\q\forall T>0.
$$

Introduce a vector field
$$\A(t,x)=\w(t,x)+\nabla\zeta(t,x),\q \text{where}\q \zeta(t,x)=\int_0^t\xi(s,x)ds.
$$
From the properties of $\w$ and $\xi$ we see that
\eq\label{cond-A}
\A,\;\p_t\A\in L^2(0,T; H^1(\O,\Bbb R^3)),\q \zeta,\;\p_t\zeta\in L^2(0,T; H^2(\O)),\q \forall T>0.
\eeq
Now we can represent $\E, \B, \H$ by $\A$ (see  \eqref{EptW}, \eqref{decom-B}, \eqref{H=HB}):
\eq\label{EB}
\left\{\aligned
&\E(t,x)=-\p_t\A(t,x)+\tilde\h_1(t,x),\\
&\B(t,x)=\curl\A(t,x)+\h_2(x),\\
&\H(t,x)=\mH(x,\curl\A(t,x)+\h_2(x)).
\endaligned\right.
\eeq
Plugging the relation $\H=\mH(x,\B)$ into \eqref{Max} (c)
we get
\eq\label{M00}
-\p_t(\var\E)+\curl [\mH(x,\curl\A+\h_2(x))]=\j.
\eeq
Finally from \eqref{Max}(a) and \eqref{EB} we have
\eq\label{pt-div-A}
-\div[\var\,(\p_t\A-\tilde\h_1)]=\rho.
\eeq

{\it Step 2}. Assume the Ohm's law \eqref{Ohm}. Then from \eqref{Max} (c) and \eqref{EB} we have
$$
-\p_t(\var\E)+\curl\H=\j=\j_a+\s\E=\j_a+\s(-\p_t\A+\tilde\h_1).
$$
Plugging the relation $\H=\mH(x,\B)$ into the above equality
we get
\eq\label{M0}
\s\p_t\A-\p_t(\var\E)+\curl [\mH(x,\curl\A+\h_2(x))]=\j_a+\s\tilde\h_1.
\eeq

\subsubsection{Quasi-static approximation}\

Next we consider a quasi-static approximation of the Maxwell equations \eqref{Max} (in the form of \eqref{M00}) by neglecting the displacement current  $\p_t(\var\E)$.

{\it Step 1. General case}.  From \eqref{M00} with the displacement current neglected we have:
\eq\label{M10}
\curl [\mH(x,\curl\A+\h_2(x))]=\j,\q t>0,\; x\in \O.
\eeq
We introduce a function $\Phi$ such that for $t\geq 0$,
\eq\label{eqPhi}
\Delta\Phi(t,x)=\div\A(t,x)\q\text{in }\O.
\eeq
Set
\eq\label{uq}
\u(t,x)=\A(t,x)-\nabla\Phi(t,x),\q  q(t,x)=-\p_t\Phi(t,x).
\eeq
From \eqref{cond-A} we have
$$\aligned
&\Phi,\; \p_t\Phi \in L^2(0,T; H^2(\O)),\q\forall T>0,\\
&\u,\; \p_t\u\in L^2(0,T; H^1(\O,\div0)),\q q\in L^2(0,T; H^2(\O)),\q\forall T>0.
\endaligned
$$
Now Eq. \eqref{M10} is transferred to
\eq\label{M20}
\left\{\aligned
&\curl [\mH(x,\curl\u+\h_2(x))]=\j,\q &t>0,\; x\in \O,\\
&\div\u=0,\q &t>0,\; x\in \O,\\
\endaligned\right.
\eeq
here $\u$ and $\j$ may depend on $t$. Eq. \eqref{pt-div-A} takes the form
\eq\label{pt-div-u}
-\div[\var\,(\p_t\u-\nabla q-\tilde\h_1)]=\rho, \q t>0,\; x\in \O.
\eeq

{\it Step 2}.
Assume the Ohm's law \eqref{Ohm}. Then a quasi-static approximation of the Maxwell equations \eqref{Max} (in the form of \eqref{M0})  gives
\eq\label{M1}
\s\p_t\A+\curl [\mH(x,\curl\A+\h_2(x))]=\j_a+\s\tilde\h_1,\q t>0,\; x\in \O.
\eeq
After introducing $\Phi(t,x)$ by \eqref{eqPhi}, and introducing $\u, \u^0, q$ by \eqref{uq}, we have \eqref{pt-div-u}, and
\eqref{M1} is transferred to
\eq\label{M2}
\left\{\aligned
&\s\p_t\u+\curl [\mH(x,\curl\u+\h_2(x))]=\j_a+\s\tilde\h_1+\s\nabla q,\q &t>0,\; x\in \O,\\
&\div\u=0,\q &t>0,\; x\in \O.
\endaligned\right.
\eeq
In the special case when $\s$ and $\var$ are positive constants, from  \eqref{pt-div-u} and \eqref{M2} we get
\eq\label{ja-rho}
-\div \j_a={\s\over\var}\rho,\q t>0,\; x\in \O.
\eeq

\subsection{Time-independent problems: magneto-static fields}\

{\it Step 1}. Assume that
\eq\label{cond-ind-t0}
\text{$\u,\; \bold j$\; are independent of $t$.}
\eeq
Then the steady state problem of \eqref{M20} takes the following form:
\eq\label{M30}
\left\{\aligned
&\curl [\mH(x,\curl\u+\h_2(x))]=\bold j\q &\text{in }\O,\\
&\div\u=0\q &\text{in }\O.
\endaligned\right.
\eeq
This equation is different to \eqref{M20} in the sense that in \eqref{M30} $\u$ and $\j$ are independent of $t$.
 \eqref{M30} should be coupled with \eqref{pt-div-u}, which takes the form
\eq\label{Lapla-p0}
\div[\var(\nabla q+\tilde\h_1)]=\rho\q\text{in }\O.
\eeq

{\it Step 2}. Assume the Ohm's law \eqref{Ohm}, and assume
\eq\label{cond-ind-t}
\text{$\u,\; \tilde\h_1,\; \bold j_a,\; q,\; \var,\; \rho$\; are independent of $t$,\;\; $\s$  is a positive constant.}
\eeq
Let us introduce $p, \h_1$ by
$$\s q(x)=p(x),\q \s\tilde\h_1(x)=\h_1(x).
$$
Since $\p_t\Phi=-q(x)=-p(x)/\s$, there exists a function $\Phi_0(x)$ such that
\eq\label{Phi-p}
\Phi(x,t)=\Phi_0(x)-{t\over\s}p(x).
\eeq
Now the relations between $\A, \E, \B, \H, q$ and  $\u, \h_1, \h_2, \Phi_0, p$ are the following:
\eq\label{EB2}
\left\{\aligned
&\A(t,x)=\u(x)+\nabla\Phi_0(x)-{t\over\s}\nabla p(x),\\
&\E(x)={1\over\s}\nabla p(x)+{1\over \s}\h_1(x),\\
&\B(x)=\curl\u(x)+\h_2(x),\\
&\H(x)=\mH(x,\curl\u(x)+\h_2(x)),\\
&q(x)={1\over\s}p(x).
\endaligned\right.
\eeq
The steady state problem of \eqref{M2} takes the following form:
\eq\label{M3}
\left\{\aligned
&\curl [\mH(x,\curl\u+\h_2(x))]=\j_a+\h_1(x)+\nabla p\q &\text{in }\O,\\
&\div\u=0\q &\text{in }\O.
\endaligned\right.
\eeq
Now the Gauss's law of electricity \eqref{pt-div-u} is reduced to \eqref{ja-rho}.
\eq\label{Lapla-p}
\div[{\var\over\s}(\nabla p+\h_1)]=\rho\q\text{in }\O,
\eeq
and \eqref{Max} is reduced to \eqref{M3}-\eqref{Lapla-p}.  $\rho$ should not be arbitrarily given, instead, it is a function determined by  $\h_1$ and $p$.

So we conclude that:

\begin{Lem}\label{Lem-conclusion3} Assume the assumptions 3.1, 3.2, 3.3.
\begin{itemize}
\item[(i)] After neglecting the displacement current $\p_t(\var\E)$ from the Maxwell equations \eqref{Max}, and introducing the magnetic potential $\u$, and assume $\u$ and $\j$ are independent of $t$, then \eqref{Max} is reduced to \eqref{M30}-\eqref{Lapla-p0}.
\item[(ii)] Assume furthermore the Ohm's law \eqref{Ohm} holds, $\s$ is a positive constant, and  assume $\u, \tilde\h_1, \j_s, q, \var, \rho$ are independent of $t$.
 \begin{itemize}
 \item[(a)] Solution $(\E,\B,\H)$ of \eqref{Max} is represented  by \eqref{EB2}, and \eqref{Max} is reduced to \eqref{M3}-\eqref{Lapla-p}.
\item[(b)] If  $\j_a$ and $\rho$ are independent of $\u$, then $\rho$ is determined by $\j_a$ by \eqref{ja-rho}, and the magneto-static problem is reduced to a single system \eqref{M3}.
\end{itemize}
\end{itemize}
\end{Lem}

\begin{Rem} Assume the Ohm's law \eqref{Ohm} and assume \eqref{cond-ind-t}.

\begin{itemize}
\item[(i)]  Equation \eqref{M3} indicates the relation between magnetic field $\H$, electric field $\E$, and electric current $\j=\s\E+\j_a$:
$$
\curl\H=\j=\j_a+\sigma \E.
$$

\item[(ii)] (About the topological effects on electromagnetism.)  \eqref{EB2} and \eqref{M3} show that the electromagnetic fields depend on $\h_1, \h_2$ which represents  topology of the domain. \eqref{EB2} indicates that  $\E$ contains $\h_1$ and $\B$ contains $\h_2$  explicitly, which suggests that $\E$ is more directly linked to simply-connectedness of $\O$, and $\H$ is more directly linked to connectedness of $\p\O$.

\item[(iii)] (About the physical meaning of $\nabla p$ in \eqref{M3}.) Mathematically we may say that, the gradient term $\nabla p$ is necessarily introduced into the equation  in \eqref{M3} to balance the equality, which is necessary for solvability of the boundary value problem. On the other hand, from \eqref{EB2} and \eqref{Ohm} we see that
$$
\nabla p=\s\E-\h_1=\bold j-\bold j_a,
$$
so $\nabla p$ represents both electric field $\E$ and domain topology.
\end{itemize}
\end{Rem}

\begin{Rem} In \cite{P4} we considered the following system
\eq\label{M4}
\curl [\mH(x,\curl\u)]=\bold j_a,\q \div\u=0\q \text{\rm in }\O.
\eeq
This is a particular case of \eqref{M3} without the topological effect terms $\h_1, \h_2$, and without the potential term $\nabla p$, where $\j_a$ is either given  or $\j_a=\f(x,\u)$.
It was shown in \cite{P4} that \eqref{M4} under some Dirichlet boundary condition may not be well-posed, and a potential term should be introduced to balance the equation:
\eq\label{M5}
\curl [\mH(x,\curl\u)]=\bold j_a+\nabla p,\q \div\u=0\q \text{\rm in }\O.
\eeq
The boundary condition for $p$ considered in \cite{P4} is Dirichlet type if $\O$   has no holes, and Neumann type if $\O$ has holes. Now we can say that $\nabla p$ gives the electric field $\E$, hence a Dirichlet boundary condition on $p$ is to prescribe the tangential component of $\E$, and a Neumann boundary condition on $p$ is to prescribe the normal component of $\E$.

\eqref{M3} is a better model than \eqref{M5} in the following sense. \eqref{M5} says that magnetic field (described by $\u$) and electric field (described by $\nabla p$) depend on each other, \eqref{M3} says that magnetic field (described by $\u$) and electric field (described by $\nabla p$ and $\h_1$) depend on each other, and both depend on domain topology (described by $\h_1, \h_2$).
\end{Rem}

\subsection{BVPs where the current is given}\

Assume $\bold j_a(x)$ is given and independent of $\u, \h_1, \h_2$.
As mentioned in Lemma \ref{Lem-conclusion3}, we should assume $\rho$ is determined by $p, \h_1$ by \eqref{Lapla-p}.
Then \eqref{Max} is reduced to a single system \eqref{M3}.
If we write $\bold j_a(x)=\J(x),$
then the first equation in \eqref{M3} can be written as follows:
$$
\curl [\mH(x,\curl\u+\h_2(x))]=\J(x)+\h_1(x)+\nabla p\q \text{in }\O.
$$
Under a suitable boundary value condition, the unknown $p(x)$ can be solved independent of $\u$. In this case, replacing $\J+\nabla p$ by $\J$, the equation is in the form of \eqref{M6}.
Then we may pose one of the following boundary conditions for $\u$ on $\p\O$:
\begin{itemize}
\item[(1)] Dirichlet boundary condition $\u_T=\u^0_T$;
\item[(2)] Tangential curl boundary condition $\nu\times\curl\u=\nu\times\B^0$;
\item[(3)] Normal curl boundary condition $\nu\cdot\curl\u+\nu\cdot\h_2(x)=B^0_n$;
\item[(4)] Natural boundary condition $\nu\times\mH(x,\curl\u+\h_2(x))=\nu\times\H^0$;
\item[(5)] Co-normal boundary condition $\nu\cdot\mH(x,\curl\u+\h_2(x))=H^0_n$;
\end{itemize}

\subsection{BVPs where the current depends on magnetic induction}\

Assume $\bold j_a$ depends on the magnetic induction $\B=\curl\u+\h_2.$
Then we may write $\bold j_a=\f(x,\curl\u+\h_2),$
and \eqref{M3} takes the form of \eqref{M7}.
We may pose one of the above mentioned boundary conditions for $\u$, and either Dirichlet or Neumann condition for $p$. Other type boundary conditions are also possible. However, following the discussions in \cite{P4}, we have the following observations.
\begin{itemize}
\item[(1)] If $\O$ has no holes, then BVP of \eqref{M7} with Dirichlet boundary condition for $p$ (and with a suitable boundary condition on $\u$) is well-posed.
\item[(2)] If $\O$ has holes, then BVP of \eqref{M7} with Neumann boundary condition for $p$ (and with a suitable boundary condition on $\u$) is well-posed, but the problem with Dirichlet boundary condition for $p$ is not well-posed.
\end{itemize}

\section{Dirichlet BVP of the Maxwell System with a Given Current}

In this section we consider the Dirichlet BVP of the Maxwell system with a given current $\J$:
\eq\label{MD-J}
\left\{\aligned
&\curl [\mH(x,\curl\u+\h_2(x))]=\J(x)+\h_1(x)\q &\text{in }\O,\\
&\div\u=0\q &\text{in }\O,\\
&\u_T=\u^0_T\q & \text{on }\p\O.
\endaligned\right.
\eeq
In \eqref{MD-J} there are no boundary conditions for $\h_1$ and $\h_2$, and we may view $\h_1, \h_2$ as parameters. If $\u$ is a solution of \eqref{MD-J}, then for any $\h\in \Bbb H_2(\O)$, $\u+\h$ is also a solution of \eqref{MD-J}. To avoid non-uniqueness due to $\Bbb H_2(\O)$, we may require the solution $\u\perp_{L^2(\O)}\Bbb H_2(\O).$
As in \cite[Lemma 2.1]{P3} we can show that the condition $\J\in \mH^\Gamma(\O,\div0)$ is necessary for \eqref{MD-J} to have a weak solution.
Throughout this section we assume that
\eq\label{cond-MJ}
\aligned
&\text{$\O$ is a bounded domain in $\Bbb R^3$ with a $C^2$ boundary,}\\
&\text{$\mH(x,\z)$ satisfies $(H_1), (H_2), (H_3)$,}\q
\bold J\in \mH^\Gamma(\O,\div0),\\
&\h_1\in\Bbb H_1(\O),\q \h_2\in\Bbb H_2(\O),
\endaligned
\eeq
and
\eq\label{condMD-J}
\u^0_T\in T\!H^{1/2}(\p\O,\Bbb R^3),\q \nu\cdot\curl\u^0_T\in H^{-1/2}(\p\O).
\eeq
\eqref{condMD-J} implies that there exists a divergence-free and curl-minimizing extension $\mU$ of $\u^0_T$, see \cite{P2}. Set $\bold b=\curl\mU$ and $\bv=\u-\mU$. Then \eqref{MD-J} can be written as follows:
\eq\label{MD-J0}
\left\{\aligned
&\curl [\mH(x,\curl\bv+\bold b+\h_2)]=\J(x)+\h_1(x)\q &\text{in }\O,\\
&\div\bv=0\q &\text{in }\O,\\
&\bv_T=\0\q & \text{on }\p\O.
\endaligned\right.
\eeq

In the following we derive  solvability of \eqref{MD-J} by variational methods, monotone operator methods, and the reduction methods. See Proposition 5.7 in \cite{P4}) for  solvability by compact operator methods.

\subsection{Solvability of \eqref{MD-J} by variational methods}\

\begin{Prop}\label{Prop-min-MD-J} Let $\O, \J, \h_1,\h_2,\u^0_T$ satisfy \eqref{cond-MJ} and \eqref{condMD-J}, $\mH(x,\z)=\nabla_\z P(x,\z)$ for a function $P$ satisfying
\begin{itemize}
\item[$(P_1)$] $P\in C^{1,\a}_{\loc}(\bar\O\times\Bbb R^3)$, $P(x,\z)$ is strictly convex in $z$, and there exist positive constants $c_1, c_2, c_3$,  functions $p_1\in L^1(\O)$ and $p_2\in L^2(\O)$ such that
$$c_1|\z|^2-p_1(x)\leq P(x,\z)\leq c_2|\z|^2+p_1(x),\q |\nabla_\z P(x,\z)|\leq c_3|\z|+p_2(x).
$$
\end{itemize}
Then  \eqref{MD-J} has a solution $\u\in H^1(\O,\div0)$.
\end{Prop}

\begin{proof}
Condition $(P_1)$ here is slightly more general than the condition $(P)$ in \cite{P4}.
By the argument in the proof of \cite[Proposition 3.1]{P4} we know that the energy functional has a minimizer $\u$ in $H^1_{t0}(\O,\div0)\cap\Bbb H_2(\O)^\perp$.
Hence for all $\w\in H^1_{t0}(\O,\div0)\cap\Bbb H_2(\O)^\perp$ it holds that
\eq\label{wk-PF}
\int_\O\{\nabla_\z P(x,\curl\bv +\bold b+\h_2)\cdot\curl\w-(\J-\h_1)\cdot\w\}dx=0.
\eeq
This equality remains true for all $\w\in H^1_{t0}(\O,\div0)$ because $\J+\h_1$ is perpendicular to $\Bbb H_2(\O)$.
So we can use a variant form of the De Rham lemma for functionals vanishing on $H^1_{t0}(\O,\div0)$ (see  \cite[Lemma 2.2]{P4} and \cite[Lemma 4.2]{P5}) to conclude that there exists $p\in L^{2,-1/2}(\O)$ such that the equality
$$\curl[\nabla_\z P(x,\curl\bold v+\bold b+\h_2)]-\J-\h_1=\nabla p
$$
holds as elements in $H^{1,*}_{t0}(\O,\Bbb R^3)$,  and  $\g(p)=0$ on $\p\O$. This means that
$$
\int_\O\{\nabla_\z P(x,\curl\bv +\bold b+\h_2)\cdot\curl\w-(\J-\h_1)\cdot\w\}dx=-\int_\O p\,\div\w dx\q\forall \w\in H^1_{t0}(\O,\Bbb R^3).
$$
Let $\zeta$ be such that
$$
\Delta \zeta=p\q\text{in }\O,\q \zeta=0\q\text{on }\p\O.
$$
Then $\w=\nabla\zeta\in H^1_{t0}(\O,\Bbb R^3)$. Plugging it into the above equality we find
$$
\int_\O p^2dx=\int_\O p\,\div\nabla\zeta dx=\int_\O(\J+\h_1)\cdot\nabla\zeta dx=0
$$
because $\div(\J+\h_1)=0$ in $\O$ and $\zeta=0$ on $\p\O$. Hence $p=0$.
So $\bold v$ is a weak solution of \eqref{MD-J0}.
\end{proof}

\begin{Cor}\label{Cor-quasi-MD-J} Let $\O, \J, \h_1,\h_2,\u^0_T$ satisfy \eqref{cond-MJ} and \eqref{condMD-J}, $\mH(x,\z)=a(x,|\z|^2)\z$, where $a(s,x)$ satisfies the following condition:
\begin{itemize}
\item[$(a)$] $a\in  C^1_{\loc}(\bar\O\times\Bbb R^3)$, $\nabla_x a\in C^\a_{\loc}(\bar\O\times\Bbb R^3,\Bbb R^3)$ with $0<\a<1$,   and there exist positive constants $\lam, \Lambda$ and $\delta\in (0,1)$ such that
\eq\label{cond-a}
\lam\leq a(t,x)\leq \Lambda,\q {\p a\over\p s}(s,x)|s|\geq -{1-\delta\over 2}a(x,z).
\eeq
\end{itemize}
Then \eqref{MD-J} has a weak solution $\u\in H^1(\O,\div0)$.
\end{Cor}

\begin{proof} The proof is similar to the proof of Corollary 3.2 in \cite{P4}.
Let
$c(t,x)=\int_0^{|t|} a(s,x)ds$, $P(x,\z)=c(x,|\z|^2)$.
From $(a)$ we see that $P(x,\z)$ satisfies $(P_1)$. In particular,
$$\aligned
&\sum_{i,j=1}^3{\p^2 P\over\p x_i\p x_j}(x,\z)\xi_i\xi_j=2a(x,|\z|^2)|\xi|^2+4{\p a\over \p s}(x,|\z|^2)(\z\cdot\xi)^2\\
&\q\geq \begin{cases} 2a(x,|\z|^2)|\xi|^2\q&\text{if } {\p a\over\p s}(x,|\z|^2)\geq 0,\\
                     2|\xi|^2\{a(x,|\z|^2)-2|{\p a\over\p s}(x,|\z|^2)||\z|^2\}\geq 2\delta|\xi|^2a(x,|\z|^2)\q&\text{if } {\p a\over \p s}(x,|\z|^2)<0.\end{cases}
\endaligned
$$
So $P(x,\z)$ is strictly convex in $\z$.
Hence the conclusion follows from Proposition \ref{Prop-min-MD-J}.
\end{proof}

\subsection{Solvability of \eqref{MD-J} by monotone operator methods}\

\begin{Thm}\label{Thm-Mono-MD-J} Assume $\O, \J, \h_1, \h_2, \u^0_T$ satisfy \eqref{cond-MJ}, \eqref{condMD-J}, and $\mH$ satisfies $(H_1), (H_2)$. Then \eqref{MD-J} has a unique solution $\u\in H^1(\O,\div0)\cap \Bbb H_2(\O)^\perp_{L^2(\O)}$. Moreover, the solution map $(\h_1,\J)\mapsto \u$ is continuous from $\Bbb H_1(\O)\times \mH^\Gamma(\O,\div0)$ to $H^1(\O,\div0)$.
\end{Thm}

\begin{proof} {\it Step 1}.
Define
\eq\label{sp-X}
X=H^1_{t0}(\O,\div0)\cap \Bbb H_2(\O)^\perp_{L^2(\O)}.
\eeq
Endowed with the norm in $H^1(\O,\Bbb R^3)$, $X$ is a separable Hilbert space.
By the div-curl-gradient inequalities we know that
$$\|\w\|_{H^1(\O)}\leq C(\O)\|\curl\w\|_{L^2(\O)},\q\forall \w\in X.
$$
Hence both $\|\w\|_{H^1(\O)}$ and $\|\curl\w\|_{L^2(\O)}$ are the equivalent norms in $X$. Denote by $X^*$ the dual space of $X$, and denote by $\langle \cdot , \cdot\rangle_{X^*,X}$ the paring between $X^*$ and $X$.

For any given $\h_2\in \Bbb H_2(\O)$, $\w\in X$, we define a functional $A_{\h_2,\w}$ on $X$ by
$$
A_{\h_2,\w}[\bv]=\int_\O \langle\mH(x,\curl\w+\bold b+\h_2), \curl\bv\rangle dx,\q \forall \bv\in X.
$$
Using the condition $(H_1)$ we see that $A_{\h_2,\w}$ is a bounded and linear functional on $X$ and
$$\aligned
|A_{\h_2,\w}[\bv]|\leq \{c_1\|\curl\w+\bold b+\h_2\|_{L^2(\O)}+\|g_2\|_{L^2(\O)}\}\|\curl\bv\|_{L^2(\O)}.
\endaligned
$$
Thus there is an element in $X^*$, which is denoted by $\mA_{\h_2}(\w)$, such that
\eq\label{mAw}
\langle \mA_{\h_2}(\w), \bv\rangle_{X^*,X}=A_{\h_2,\w}[\bv]=\int_\O \langle\mH(x,\curl\w+\bold b+\h_2), \curl\bv\rangle dx,\q \forall \bv\in X.
\eeq
The mapping $\w\mapsto \mA_{\h_2}(\w)$ defines an operator
$\mA_{\h_2}: X\to X^*$, and we have
$$
\|\mA_{\h_2}(\w)\|_{X^*}\leq c_1\|\curl\w+\bold b+\h_2\|_{L^2(\O)}+\|g_2\|_{L^2(\O)}.
$$
From $(H_2)$ we know that $\mH(x,\z)$ is Lipschitz in $\z$, hence $\mA_{\h_2}$ is continuous. On the other hand, $(H_2)$ implies $(H_4)$,  hence
$$\aligned
\langle \mA_{\h_2}(\w_2)-\mA_{\h_2}(\w_1), \w_2-\w_1\rangle_{X^*,X}
\geq \mu\int_\O|\curl(\w_2-\w_1)|^2dx\geq \mu C_1(\O)\|\w_2-\w_1\|_{X}^2.
\endaligned
$$
Hence $\mA_{\h_2}$ is strongly monotone from $X$ to $X^*$. By the  Browder-Minty theorem (see for instance \cite[p.557, Theorem 26.A]{Z2}) we know that $\mA_{\h_2}$ is surjective.

{\it Step 2}. Given any $\u\in L^2(\O,\Bbb R^3)$ we can define a functional $L[\u]$ on $X$ by
$$
L[\u](\bv)=\int_\O \u\cdot\bv dx,\q \forall \bv\in X.
$$
Then $L[\u]$ is a bounded and linear functional on $X$, and there is an element in $X^*$, which is denoted by $\mL(\u)$, such that
$$
\langle \mL(\u),\bv\rangle_{X^*,X}=L[\u](\bv).
$$
Since $\mA_{\h_2}$ is surjective and $\mL(\J+\h_1)\in X^*$, there exists $\w\in X$ such that
$\mA_{\h_2}(\w)=\mL(\J+\h_1)$, that is,
\eq\label{mAJ}
\int_\O \langle\mH(x,\curl\w+\bold b+\h_2), \curl\bv\rangle dx=\int_\O(\J+\h_1)\cdot\bv dx,\q \forall \bv\in X.
\eeq
From \eqref{mAw} and since $\mH^\Gamma(\O,\div0)\subset \Bbb H_2(\O)^\perp_{L^2(\O)}$ we have
$$
\langle \mA_{\h_2}(\w),\h\rangle_{X^*,X}=0,\qq \int_\O(\J+\h_1)\cdot\h dx=0,\q\forall \h\in \Bbb H_2(\O).
$$
From these and \eqref{mAJ} we have
\eq\label{mAJ2}
\int_\O \langle\mH(x,\curl\w+\bold b+\h_2), \curl\bv\rangle dx=\int_\O(\J+\h_1)\cdot\bv dx,\q \forall \bv\in H^1_{t0}(\O,\div0).
\eeq
Hence by the De Rham lemma on $H^1_{t0}(\O,\div0)$ (see \cite[Lemma 2.2]{P4} and \cite[Lemma 4.2]{P5}) we know that there exists $p\in L^{2,-1/2}_0(\O)$ such that, in the sense of functionals in $H^1_{t0}(\O,\div0)$ we have
$$\curl [\mH(x,\curl\w+\bold b+\h_2)]=\J+\h_1+\nabla p.
$$
Using the argument in the proof of Proposition \ref{Prop-min-MD-J} we can show that $p=0$. Thus $\w\in H^1_{t0}(\O,\div0)$ is a weak solution of \eqref{MD-J0}.
By the strongly monotonicity of $\mA_{\h_2}$ we know that for any $\J+\h_1$, the weak solution $\w$ is unique, and the map $\J+\h_1\mapsto \w$ is continuous.
\end{proof}

\subsection{Solvability of \eqref{MD-J} by the reduction method}\

A description of the reduction method can be found in \cite{P3, P4}. Since $\J+\h_1\in \mH^\Gamma(\O,\div0)$, there exists $\bold j_0\in H^1(\O,\Bbb R^3)$ such that
\eq\label{eq-j}
\left\{\aligned
&\curl\bold j_0=\J+\h_1,\q \div\bold j_0=0\q&\text{in }\O,\\
&\nu\cdot\bold j_0=0\q&\text{on }\p\O,\\
&\bold j_0\perp_{L^2(\O)}\Bbb H_1(\O).
\endaligned\right.
\eeq
such a $\bold j_0$ is unique. The following lemma is obvious.

\begin{Lem}\label{Lem1-equiv-MD-J} Assume $\O, \J, \h_1, \h_2$ satisfy \eqref{cond-MJ} and \eqref{condMD-J}, and $\bold j_0$ is the solution of \eqref{eq-j}.
\begin{itemize}
\item[(i)] System \eqref{MD-J0} has a weak solution $\bv\in H^1_{t0}(\O,\div0)$ if and only if there exists $\phi\in H^1(\O)$ and $\h_1'\in \Bbb H_1(\O)$  such that the following system  has a solution $\bv$:
\eq\label{eqH}
\left\{\aligned
&\mH(x,\curl\bv+\curl\mU+\h_2)=\bold j_0+\h_1'+\nabla\phi &\text{\rm in }\O,\\
&\div\bv=0 &\text{\rm in }\O,\\
&\bv_T=\0&\text{\rm on }\p\O.
\endaligned\right.
\eeq

\item[(ii)] If furthermore $\mH$ satisfies $(H_3)$, then \eqref{eqH} is equivalent to the following
\eq\label{eqB}
\left\{\aligned
&\curl\bv=\mB(x,\bold j_0+\h_1'+\nabla\phi)-\curl\mU-\h_2 &\text{\rm in }\O,\\
&\div\bv=0 &\text{\rm in }\O,\\
&\bv_T=\0&\text{\rm on }\p\O.
\endaligned\right.
\eeq
\end{itemize}
\end{Lem}

\begin{Lem}\label{Lem-reduction-MD-J} Given $\h_1'\in\Bbb H_1(\O)$, system \eqref{eqB} has a solution $\bold v\in H^1_{t0}(\O,\div0)$ if and only if the following equation
\eq\label{eqdivB-MD-J}
\left\{\aligned
&\div[\mB(x,\bold j_0+\h_1'+\nabla\phi)]=0\q&\text{\rm in }\O,\\
&\nu\cdot \mB(x,\bold j_0+\h_1'+\nabla\phi)=\nu\cdot(\h_2+\curl\u^0_T)\q&\text{\rm on }\p\O,\\
\endaligned\right.
\eeq
has a solution $\phi\in H^1(\O)$ which satisfies the following orthogonality condition
\eq\label{orth-MD-J-B}
\int_\O\langle\mB(x,\bold j_0+\h_1'+\nabla\phi)-\curl\mU,\h\rangle dx=0,\q\forall \h\in\Bbb H_1(\O).
\eeq
\end{Lem}

\begin{proof} Recall that \eqref{eqB} has a solution if and only if
\eq\label{BinHsigma}
\mB(x,\bold j_0+\h_1'+\nabla\phi)-\curl\mathcal U-\h_2 \in \curl H^1_{t0}(\O,\div0)= \mathcal H^\Sigma_0(\O,\div0).
\eeq
On the other hand,  $\phi$ satisfies \eqref{eqdivB-MD-J} if and only if
$$\mB(x,\bold j_0+\h_1'+\nabla\phi)-\curl\mU-\h_2\in \mH_0(\O,\div0)=\mH^\Sigma_0(\O,\div0)\oplus_{L^2(\O)}\Bbb H_1(\O).
$$
Here we have used the equality
$\nu\cdot\curl\mathcal U=\nu\cdot\curl\mathcal U_T
=\nu\cdot\curl\u^0_T.$
Hence \eqref{BinHsigma} holds if and only if $\phi$ satisfies \eqref{eqdivB-MD-J} and
$$\mB(x,\bold j_0+\h_1'+\nabla\phi)-\curl\mathcal U -\h_2\perp_{L^2(\O)}\Bbb H_1(\O).
$$
Since $\h_2\in\Bbb H_2(\O)\subset \Bbb H_1(\O)^\perp_{L^2(\O)}$, the above orthogonality is equivalent to \eqref{orth-MD-J-B}.
\end{proof}

\begin{Lem}\label{Lem-uniq} Assume $\mH$ satisfies $(H_3)$ and $(H_4)$. Then for any given  $\j_0$, $\mU$ and $\h_2$, there exists at most one  $(\phi,\h_1')\in \dot H^1(\O)\times\Bbb H_1(\O)$ such that \eqref{eqB} is solvable.
\end{Lem}

\begin{proof}
Suppose there exist $\phi_j\in \dot H^1(\O)$ and $\h'_{1,j}\in \Bbb H_1(\O)$, $j=1,2$, such that \eqref{eqB} with $\phi=\phi_j$ and $\h_1'=\h_{1,j}'$ has a solution $\bv_j$. We may assume $\bv_j\in \Bbb H_2(\O)^\perp_{L^2(\O)}$.
Write $\u_j=\bv_j+\mathcal U$. Then
$$
\mH(x,\curl\u_2+\h_2)-\mH(x,\curl\u_1+\h_2)=\nabla(\phi_2-\phi_1)+\h'_{1,2}-\h'_{1,1}.
$$
Multiplying it by $\curl(\u_2-\u_1)$ and integrating, and using condition $(H_4)$ we have
$$\aligned
&\mu\|\curl(\u_2-\u_1)\|_{L^2(\O)}^2\\
\leq& \int_\O\langle \mH(x,\curl\u_2+\h_2)-\mH(x,\curl\u_1+\h_2),\curl(\u_2-\u_1)\rangle dx\\
=&\int_\O\langle \h'_{1,2}-\h'_{1,1}+\nabla(\phi_2-\phi_1), \curl(\u_2-\u_1)\rangle dx\\
=&\int_{\p\O}\langle\nu\times(\u_{2T}-\u_{1T}),\h'_{1,2}-\h'_{1,1}+\nabla(\phi_2-\phi_1))\rangle dS=0.
\endaligned
$$
In the last line we used the fact that $\nu\times\u_{2T}=\nu\times\u_{1T}=\nu\times \u^0_T$. So
$\curl(\u_2-\u_1)=\0$, hence $\u_2-\u_1\in \Bbb H_2(\O)$.
This together with the fact $\u_2-\u_1=\bv_2-\bv_1\in \Bbb H_2(\O)^\perp_{L^2(\O)}$ implies that
$\u_2-\u_1=\0$, so $\bv_2-\bv_1=\0$.

Finally, since $\u_2=\u_1$,
$$\h_{1,2}'-\h_{1,1}'+\nabla(\phi_2-\phi_1)=\mH(x,\curl\u_2+\h_2)-\mH(x,\curl\u_1+\h_2)=\0.
$$
Since $\h_{1,2}'-\h_{1,1}'\in \Bbb H_1(\O)$, and $\nabla(\phi_2-\phi_1)\in \nabla H^1(\O)$ which is orthogonal to $\Bbb H_1(\O)$, we conclude that $\h_{1,2}'-\h_{1,1}'=\0$ and $\nabla(\phi_2-\phi_1)=\0$.
\end{proof}

\subsubsection{Existence of solutions of \eqref{eqdivB-MD-J} by the monotone operator method}\

\eqref{eqdivB-MD-J} is a quasilinear co-normal problem. Existence and regularity of weak solutions to such type problems have been well studied, see for instance \cite[Chapter 10]{LU}, \cite[Theorem 11.20]{Li}.

\begin{Prop}\label{Prop-sol}  Assume that $\O, \mH, \h_2, \u^0_T$ satisfy \eqref{cond-MJ} and $\j_0\in H^1(\O,\Bbb R^3)$.
\begin{itemize}
\item[(i)] For any $\h_1'\in \Bbb H_1(\O)$, \eqref{eqdivB-MD-J} has a unique solution $\phi\in \dot H^1(\O)$.
\item[(ii)] There exists $\h_1'\in \Bbb H_1(\O)$ such that \eqref{eqdivB-MD-J} has a solution $\phi\in \dot H^1(\O)$ which satisfies \eqref{orth-MD-J-B}.
\end{itemize}
\end{Prop}

\begin{proof} (i). As in the proof of Theorem \ref{Thm-Mono-MD-J}, we fix $\h_1'\in\Bbb H_1(\O)$ and define a map $\mT$ such that
$$
\langle \mT(\phi),\eta\rangle_{\dot H^1(\O)^*, \dot H^1(\O)}=\int_\O \langle\mB(x,\j_0+\h_1'+\nabla\phi), \nabla\eta\rangle dx,\q \forall \phi, \eta\in \dot H^1(\O),
$$
 Using condition $(B_1)$, $(B_2)$ we can show that $\mT: \dot H^1(\O)\mapsto \dot H^1(\O)^*$ is strongly monotone. By the Browder-Minty theorem $\mT$ is surjective, hence there exists a unique $\phi\in \dot H^1(\O)$ such that, for any $\eta\in \dot H^1(\O)$,
$$
\langle \mT(\phi),\eta\rangle_{\dot H^1(\O)^*, \dot H^1(\O)}=\langle  \nu\cdot (\h_2+\curl\A^0_T),\eta\rangle_{\p\O,1/2}.
$$
Then $\phi$ is the weak solution of \eqref{eqdivB-MD-J}.

(ii). Choose an orthonormal basis $\{\e_1,\cdots,\e_N\}$ of $\Bbb H_1(\O)$, and write
$\xi=(\xi_1,\cdots,\xi_N)^t$, $\h_1'=\h'_{1,\xi}=\sum_{i=1}^N\xi_i\e_i$,
and write the unique solution of \eqref{eqdivB-MD-J} associated with $\h'_{1,\xi}$ by $\phi=\phi_{\xi}$.

{\it Step 2.1}. We show that the map $\xi\mapsto \phi_\xi$
is continuous from $\Bbb R^N$ into $\dot H^1(\O)$. To prove, let $\xi_0\in\Bbb R^N$ and assume $\xi_j\to \xi_0$. Since $\Bbb H_1(\O)$ is of finite dimensions, and $\xi_j\to\xi_0$, so $\bold v_j=\h'_{1,\xi_j}-\h'_{1,\xi_0}\to 0$ in $H^1(\O,\Bbb R^3)$.
Let us denote
\eq\label{B1}
\mB_1(\w)=\mB(x,\bold j_0+\w).
\eeq
Set $\psi_j=\phi_{\xi_j}-\phi_{\xi_0}$, $\bold v_j=\h'_{1,\xi_j}-\h'_{1,\xi_0}$, $\w_0=\h'_{1,\xi_0}+\nabla\phi_{\xi_0}$.
We have
$$
\aligned
0=&\int_\O\langle \mB_1(\h'_{1,\xi_j}+\nabla\phi_{\xi_j})- \mB_1(\h'_{1,\xi_0}+\nabla\phi_{\xi_0}), \nabla(\phi_{\xi_j}-\phi_{\xi_0}\rangle dx\\
=&\int_\O\int_0^1\langle \nabla_{\w}\mB_1(\w_0 +t(\bold v_j+\nabla\psi_j))[\bold v_j+\nabla\psi_j],\nabla\psi_j \rangle dx dt\\
\geq &\lam_0\int_\O|\nabla\psi_j|^2 dx-N\int_\O|\bold v_j||\nabla\psi_j|dx.\\
\endaligned
$$
Hence $\|\nabla\psi_j\|_{L^2(\O)}\leq {N\over\lam_0}\|\bold v_j\|_{L^2(\O)}\to 0.$

{\it Step 2.2}. We can further show that there exists a constant $C(\O)$, which depends also on the constants in $(B_1), (B_2)$, such that
\eq\label{H1est}
\|\phi_\xi\|_{H^1(\O)}\leq C(\O)\{\|\nu\cdot(\h_2+\curl\u^0_T)\|_{H^{-1/2}(\p\O)}+\|\j_0+\h'_{1,\xi}\|_{L^2(\O)}+\|g_1\|_{L^2(\O)}\}.
\eeq
In fact, using the definition of weak solutions to \eqref{eqdivB-MD-J} we have
$$
\int_\O\langle \mB_1(\h'_{1,\xi}+\nabla\phi_\xi)-\mB_1(\h'_{1,\xi}), \nabla\phi_\xi\rangle dx=\langle\nu\cdot(\h_2+ \curl\u^0_T),\phi_\xi\rangle_{\p\O,1/2}
-\int_\O\langle \mB_1(\h'_{1,\xi}),\nabla\phi_\xi\rangle dx.
$$
Using $(B_1)$ and $(B_2)$ we get from the above inequality that
$$\aligned
\lam_0\|\nabla\phi_\xi\|_{L^2(\O)}^2\leq & \|\phi_\xi\|_{H^{1/2}(\p\O)}\|\nu\cdot(\h_2+\curl\u^0_T)\|_{H^{-1/2}(\p\O)}+\|\mB_1(\h'_{1,\xi})\|_{L^2(\O)}\|\nabla\phi_\xi\|_{L^2(\O)}\\
\leq &C(\O)\|\nabla\phi_\xi\|_{L^2(\O)}\|\nu\cdot(\h_2+\curl\u^0_T)\|_{H^{-1/2}(\p\O)}+\|\mB_1(\h'_{1,\xi})\|_{L^2(\O)}\|\nabla\phi_\xi\|_{L^2(\O)}.
\endaligned
$$
From this and $(B_1)$ we get \eqref{H1est}.

{\it Step 2.3}. We claim that
\eq\label{claim1}
\liminf_{|\xi|\to\infty}{\|\h'_{1,\xi}+\nabla\phi_\xi\|_{L^2(\O)}^2\over|\xi|^2}\geq 1.
\eeq
Since $\h'_{1,\xi}\in \Bbb H^1(\O)$, so $\h'_{1,\xi}$ and $\nabla\phi_\xi$ are orthogonal to each other with respect to the $L^2$ inner product. So $$\|\h'_{1,\xi}+\nabla\phi_\xi\|_{L^2(\O)}^2=\|\h'_{1,\xi}\|_{L^2(\O)}^2+\|\nabla\phi_\xi\|_{L^2(\O)}^2\geq \|\h'_{1,\xi}\|_{L^2(\O)}^2=|\xi|^2.
$$
So \eqref{claim1} is true.

{\it Step 2.4}. Using the basis $\{\e_1,\cdots,\e_N\}$ of $\Bbb H_1(\O)$, \eqref{orth-MD-J-B} can be written as
\eq\label{fccc}
\int_\O\langle\mB_1(\h'_{1,\xi}+\nabla\phi_\xi)-\curl\mU,\e_i\rangle dx=0,\q i=1,\cdots, N.
\eeq
Let us write
$$\aligned
& c_i=\int_\O\e_i\cdot\curl\mU dx,\qq \bold c=(c_1,\cdots,c_N)^t,\\
& f_i(\xi)=\int_\O\langle\mB_1(\h'_{1,\xi}+\nabla\phi_\xi),\e_i\rangle dx,\qq \f(\xi)=(f_1(\xi),\cdots, f_N(\xi))^t.
\endaligned
$$
Then \eqref{fccc} can be written as
\eq\label{fc}
\f(\xi)=\bold c.
\eeq
Since the map $\xi\mapsto \phi_\xi$ is continuous from $\Bbb R^N$ to $\dot H^1(\O)$, and the map
$\w\mapsto \mB(x,\w(x))$ is continuous from $L^2(\O,\Bbb R^3)$ to $L^2(\O,\Bbb R^3)$, so the map
$\xi\mapsto \mB_1(\h'_{1,\xi}+\nabla\phi_\xi)$
is continuous from $\Bbb R^N$ to $L^2(\O,\Bbb R^3)$. Thus $\f(\xi)$ is continuous from $\Bbb R^N$ to $\Bbb R^N$.

{\it Step 2.5}.
For any $\xi\in \Bbb R^N$ we have
$$\aligned
\langle\f(\xi),\xi\rangle=&\sum_{i=1}^N\xi_i\int_\O\langle\mB_1(\h'_{1,\xi}+\nabla\phi_\xi),\e_i\rangle dx
=\int_\O\langle\mB_1(\h'_{1,\xi}+\nabla\phi_\xi),\h'_{1,\xi}\rangle dx\\
=&\int_\O\langle\mB_1(\h'_{1,\xi}+\nabla\phi_\xi),\h'_{1,\xi}+\nabla\phi_\xi\rangle dx -\langle\nu\cdot(\h_2+ \curl\u^0_T), \phi_\xi\rangle_{\p\O,1/2}\\
=&\int_\O\langle\mB_1(\h'_{1,\xi}+\nabla\phi_\xi)-\B_0,\h'_{1,\xi}+\nabla\phi_\xi\rangle dx
+\int_\O\langle\B_0,\h'_{1,\xi}+\nabla\phi_\xi\rangle dx\\
&\qq  -\langle\nu\cdot(\h_2+ \curl\u^0_T),\phi_\xi\rangle_{\p\O,1/2}\\
\geq &\lam_0\int_\O|\h'_{1,\xi}+\nabla\phi_\xi|^2dx-\|\B_0\|_{L^2(\O)}\|\h'_{1,\xi}+\nabla\phi_\xi\|_{L^2(\O)}\\
&\qq -\|\phi_\xi\|_{H^{1/2}(\p\O)}\|\nu\cdot(\h_2+\curl\u^0_T)\|_{H^{-1/2}(\p\O)}\\
\geq &{3\lam_0\over 4}\|\h'_{1,\xi}+\nabla\phi_\xi\|_{L^2(\O)}^2-C_1\|\phi_\xi\|_{H^{1/2}(\p\O)}-C_2,
\endaligned
$$
where $\B_0=\mB(x,\bold j_0)$, $C_1=\|\B_0\|_{L^2(\O)}$ and $C_2=\|\nu\cdot(\h_2+\curl\u^0_T)\|_{H^{-1/2}(\p\O)}$. From this, \eqref{H1est},  \eqref{claim1} and the Cauchy inequality we get
\eq\label{fxixi}
\langle \f(\xi),\xi\rangle\geq {\lam_0\over 2}\|\h'_{1,\xi}+\nabla\phi_\xi\|_{L^2(\O)}^2-C_3.
\eeq

Using \eqref{claim1} and \eqref{fxixi} we conclude that there exists $R_0>0$ such that
$$\langle  \f(\xi),\xi\rangle\geq {\lam_0\over 3}|\xi|^2-C_4,\q\forall |\xi|\geq R_0.
$$
It follows that there exists $R>0$ such that the following acute angle condition is satisfied:
$$
\langle \f(\xi)-\bold c,\xi\rangle\geq 0,\q\forall \xi\in S_R.
$$
Here we denote $S_R=\{\xi\in \Bbb R^N: |\xi|=R\}$ and $B_R=\{\xi\in\Bbb R^N: |\xi|\leq R\}$. It then follows that there exists $\xi_0\in B_R$ such that $\f(\xi_0)-\bold c=\0$ (see Proposition 2.8 of \cite{Z1}, and Theorem
1.11 of \cite{G}). This $\xi_0$ solves equation \eqref{fc}.

For this $\xi_0$, let $\phi_{\xi_0}$ be the solution of \eqref{eqdivB-MD-J} associated with $\h'_{1,\xi_0}$. Then $(\h'_{1,\xi_0},\phi_{\xi_0})$ solves problem \eqref{eqdivB-MD-J}-\eqref{orth-MD-J-B}.
\end{proof}

Summarizing the above discussions we get the following

\begin{Thm} Assume \eqref{cond-MJ} and \eqref{condMD-J}.
Then \eqref{MD-J} has a weak solution $\u\in H^1(\O,\div0)$.
\end{Thm}

\begin{proof} Let $\j_0$ be determined by $\J$ and $\h_1$ by \eqref{eq-j}.  By Proposition \ref{Prop-sol} (ii) we know that we can choose $\h_1'$ (depending on $\h_1, \h_2$) such that \eqref{eqdivB-MD-J} has a solution $\phi=\phi_{\h_1'}\in \dot H^1(\O)$ which satisfies the orthogonality condition \eqref{orth-MD-J-B}. Then by Lemmas \ref{Lem1-equiv-MD-J} and \ref{Lem-reduction-MD-J} we know that for the given $\h_2$ and for the above $\h_1'$, \eqref{eqB} has a solution $\bold v\in H^1_{t0}(\O,\div0)$, hence \eqref{eqH} has a solution $\u\in H^1(\O,\div0)$, and $\u$ is actually a solution of \eqref{MD-J}.
\end{proof}

\section{Tangential Curl BVP of the Maxwell System with a Given Current}

In this section we consider the tangential curl BVP of the Maxwell system with a given current $\J$:
\eq\label{MC-J}
\left\{\aligned
&\curl [\mH(x,\curl\u+\h_2(x))]=\J(x)+\h_1(x)\q &\text{in }\O,\\
&\div\u=0\q &\text{in }\O,\\
&\nu\times\curl\u=\nu\times\B^0\q& \text{on }\p\O.
\endaligned\right.
\eeq
We assume condition \eqref{cond-MJ} and
\eq\label{condMC-J}
\nu\times\B^0\in T\!H^{1/2}(\p\O,\Bbb R^3).
\eeq
We shall look for weak solutions $\u\in \mH(\O,\curl,\div0)$.
If $\u_0$ is a solution, then the general solution can be written as
$\u=\u_0+\h_1'+\nabla\phi$ for any $\h_1'\in \Bbb H_1(\O)$ and $\phi\in H^1(\O)\cap\text{ker}(\Delta)$.
To avoid non-uniqueness we may look for a solution in $H^1_{n0}(\O,\div0)\cap\Bbb H_1(\O)^\perp_{L^2(\O)}$.

Denote by  $\g_\tau$ the tangential trace map $\u\mapsto \nu\times\u$. It yields a tangential trace map
\eq\label{T1T2}
T_1=\g_\tau|_{\mH^\Gamma(\O,\div0)}: \mH^\Gamma(\O,\div0)\mapsto T_1[\mH^\Gamma(\O,\div0)].
\eeq
A necessary condition for \eqref{MC-J} to have a weak solution is $\J\in \mH^\Gamma(\O,\div0)$ and
\eq\label{condJB}
\nu\times\B^0\in T_1[\mH^\Gamma(\O,\div0)].
\eeq
In fact, if \eqref{MC-J} has a weak solution $\u\in H^1(\O,\Bbb R^3)$, then $\curl\u\in \mH^\Gamma(\O,\div0)$, hence
$\nu\times\B^0=\nu\times \curl\u=T_1(\curl\u)\in T_1[\mH^\Gamma(\O,\div0)].$

\begin{Rem}\label{Rem-MCJ-1}
{\rm (i)} If $\nu\times\B^0$ has a  divergence-free $L^2$ extension to $\O$, namely if there exists $\B\in\mH(\O,\div0)$ such that $\nu\times\B=\nu\times\B^0$ on $\p\O$, then \eqref{condJB} is satisfied.

{\rm (ii)} If $\O$ is a bounded domain in $\Bbb R^3$ with a $C^2$ boundary and $\nu\times\B^0\in T\!H^{1/2}(\p\O,\Bbb R^3)$, then \eqref{condJB} is satisfied.
\end{Rem}

\begin{proof} (i) If $\nu\times\B^0$ has an $L^2$ extension $\B\in \mH(\O,\div0)$, we can write
$\B=\B_1+\h$ with $\B_1\in \mH^\Gamma(\O,\div0)$ and $\h\in\Bbb H_2(\O)$.
Then $\nu\times\B^0=\nu\times\B=\nu\times\B_1\in T_1[\mH^\Gamma(\O,\div0)].$

(ii) If $\p\O\in C^2$ and $\nu\times\B^0\in T\!H^{1/2}(\p\O,\Bbb R^3)$, then $\B^0_T=(\nu\times\B^0)\times\nu\in T\!H^{1/2}(\p\O,\Bbb R^3)$. So there exists a tangential component-preserving and divergence-free extension $\B\in \mH(\O,\div0)$ such that $\B_T=\B^0_T$, see \cite{P2}. Then $\nu\times\B=\nu\times\B_T=\nu\times\B^0_T=\nu\times\B^0$. Then from (i) we see that $\nu\times\B^0$ satisfies \eqref{condJB}.
\end{proof}

Assume  $\nu\times\B^0$ satisfies \eqref{condJB}. Then there exists $\bold b\in \mH^\Gamma(\O,\div0)$ such that
$\nu\times\bold b=\nu\times\B^0$, and  we can find $\U\in H^1_{n0}(\O,\div0)$ such that $\curl\U=\bold b$.
Let $\w=\u-\U$. Then \eqref{MC-J} is reduced to the following system
\eq\label{MC-J1}
\left\{\aligned
&\curl [\mH(x,\curl\w+\bold b(x)+\h_2(x))]=\J(x)+\h_1(x)\q &\text{in }\O,\\
&\div\w=0\q &\text{in }\O,\\
&\nu\times\curl\w=\0\q& \text{on }\p\O.
\endaligned\right.
\eeq

\subsection{Reduce \eqref{MC-J} to a nonlinear first order div-curl system \eqref{MC-J2}-\eqref{condZ}}\

\begin{Lem}\label{Lem-equiv1-MC-J} Let $\O, \J, \h_1, \h_2$ satisfy \eqref{cond-MJ} and $\bold b\in \mH^\Gamma(\O,\div0)$.
\begin{itemize}
\item[(i)] If $\w\in H^1(\O,\div0)$ is a weak solution of \eqref{MC-J1}, and let $\Z=\curl\w$, then
$\Z$ is a weak solution of
\eq\label{MC-J2}
\left\{\aligned
&\curl [\mH(x,\Z+\bold b(x)+\h_2(x))]=\J(x)+\h_1(x)\q &\text{\rm in }\O,\\
&\div\Z=0\q &\text{\rm in }\O,\\
&\nu\times\Z=\0\q& \text{\rm on }\p\O,
\endaligned\right.
\eeq
and
\eq\label{condZ}
\Z\in \mH^\Gamma(\O,\div0)\cap L^{2,-1/2}_{t0}(\O,\Bbb R^3).
\eeq

\item[(ii)] If \eqref{MC-J2} has a solution $\Z$ with
\eq\label{condZ1}
\Z\in \mH(\O,\div0)\cap L^{2,-1/2}_{t0}(\O,\Bbb R^3),
\eeq
then there exist $\hat\h_2\in \Bbb H_2(\O)$ and $\bold w\in H^1_{n0}(\O,\div0)$, such that $\curl\w=\Z-\hat\h_2$, and $\w$ is a weak solution of \eqref{MC-J1} with $\h_2$ replaced by $\h_2+\hat\h_2$.
If furthermore $\Z$ satisfies \eqref{condZ}, then $\hat\h_2=\0$, hence $\w$ is a solution of \eqref{MC-J1} with the given $\h_2$.
\end{itemize}
\end{Lem}

\begin{proof} (i) is obvious. We prove (ii).
If \eqref{MC-J2} has a solution $\Z$ satisfying \eqref{condZ1},
then since $\div\Z=0$, using the second equality of \eqref{dec-0} we can write
$\Z=\bold g+\hat\h_2$ with $\bold g\in \mH^\Gamma(\O,\div0)$ and $\hat\h_2\in \Bbb H_2(\O)$.
Take $\bold w\in H^1_{n0}(\O,\div0)$ such that $\curl\bold w=\bold g$. Then $\bold w$ satisfies
$$
\left\{\aligned
&\curl [\mH(x,\curl \bold w+\bold b(x)+\h_2'(x))]=\J(x)+\h_1(x)\q &\text{in }\O,\\
&\div\bold w=0\q &\text{in }\O,\\
&\nu\times\curl\bold w=\0\q& \text{on }\p\O,
\endaligned\right.
$$
where $\h_2'=\h_2+\hat\h_2\in \Bbb H_2(\O)$. Thus,  \eqref{MC-J} with the given $\h_2$ replaced by $\h_2'$ has a solution $\u=\w+\U$.
If $\Z$  satisfies \eqref{condZ}, then $\hat\h_2=\0$, and  $\curl\bold w=\Z$. Hence $\u=\bold w+\U$ is  a solution of \eqref{MC-J1} for the given $\h_2$.
\end{proof}

\begin{Lem}\label{Lem-equiv2-MC-J}
Assume $\O, \J, \h_1, \h_2$ satisfy \eqref{cond-MJ}, $\mH(x,\z)$ is of $C^1$ and satisfies $(H_3)$,  and
 $\bold b\in C^1(\bar\O,\Bbb R^3)\cap \mH^\Gamma(\O,\div0)$ satisfying $\nu\times\bold b=\nu\times\B^0$.
 \begin{itemize}
 \item[(i)] If $\Z$ is a $C^1$ solution of \eqref{MC-J2} and set $\Y(x)=\mH(x,\Z(x)+\bold b(x)+\h_2(x))$,
 then $\Y(x)$ is a $C^1$ solution of
 \eq\label{MC-J4}
\left\{\aligned
&\curl \Y=\J(x)+\h_1(x)\q &\text{\rm in }\O,\\
&\div\mB(x,\Y)=0\q &\text{\rm in }\O,\\
&\nu\times\mB(x,\Y)=\nu\times\B^0\q& \text{\rm on }\p\O.
\endaligned\right.
\eeq
  \item[(ii)] If $\Y(x)$ is a $C^1$ solution of \eqref{MC-J4} and set
  \eq\label{bd-Z}
\Z(x)=\mB(x,\Y(x))-\bold b(x)-\h_2(x).
\eeq
then $\Z$ is a $C^1$ solution of \eqref{MC-J2}. Moreover, $\Z\in \mH^\Gamma(\O,\div0)$ if and only if
\eq\label{orth6.9}
\mB(x,\Y(x))-\bold b(x)-\h_2\perp_{L^2(\O)}\Bbb H_2(\O).
\eeq
\end{itemize}
\end{Lem}

\begin{proof} (i) is obvious and we prove (ii). If $\Y$ is a solution of \eqref{MC-J4} and set $\Z$ by \eqref{bd-Z}, then $\Z$ is of $C^1$ and
$\mH(x,\Z(x)+\bold b(x)+\h_2(x))=\Y$. Using this and \eqref{MC-J4} we see that $\Z$ is a solution of \eqref{MC-J2}.
Since $\Z\in \mH(\O,\div0)$, from the second equality of \eqref{dec-0} we see that, $\Z\in \mH^\Gamma(\O,\div0)$ if and only if \eqref{orth6.9} holds.
\end{proof}

\subsection{Reduce \eqref{MC-J2} to a second order scalar equation \eqref{eqdivB-MC-J}}\

\begin{Lem}\label{Lem-equiv3-MC-J} Assume $\O, \J, \h_1, \h_2$ satisfy \eqref{cond-MJ}, $\mH(x,\z)$ satisfies $(H_3)$ and $\bold b\in C^1(\bar\O,\Bbb R^3)\cap \mH^\Gamma(\O,\div0)$ satisfying $\nu\times\bold b=\nu\times\B^0$. Then we can find $\Y^0$ satisfying
\eq\label{Y0}
\Y^0\in H^1_{n0}(\O,\div0),\q  \curl \Y^0=\J+\h_1,\q \Y^0\perp_{L^2(\O)}\Bbb H_1(\O).
\eeq
\begin{itemize}
\item[(i)] If \eqref{MC-J4} has a solution $\Y\in \mH(\O,\curl)$, then $\Y$ can be written as
\eq\label{form-Y}
\Y=\Y^0+\h^0_1+\nabla\phi,\q \h_1^0\in\Bbb H_1(\O),\q\phi\in H^1(\O),
\eeq
and  $\phi$ is a weak solution of the following problem
\eq\label{eqdivB-MC-J}
\left\{\aligned
&\div [\mB(x,\Y^0+\h^0_1+\nabla\phi)]=0\q&\text{\rm in }\O,\\
&\nu\times\mB(x,\Y^0+\h^0_1+\nabla\phi)=\nu\times\B^0\q&\text{\rm on }\p\O.
\endaligned\right.
\eeq
\item[(ii)] If there exists $\h_1^0\in \Bbb H_1(\O)$ such that \eqref{eqdivB-MC-J} has a solution $\phi\in H^1(\O)$, and set $\Y$ by \eqref{form-Y}, then $\Y$ is a solution of \eqref{MC-J4}.
\end{itemize}
\end{Lem}

\begin{proof} We only need to prove (i). Assume \eqref{MC-J4} has a solution $\Y\in \mH(\O,\curl)$. Then
$$
\J+\h_1=\curl\Y\in \curl\mH(\O,\curl)=\curl H^1_{n0}(\O,\div0)= \mH^\Gamma(\O,\div0).
$$
So there exists a unique $\Y^0$ satisfying \eqref{Y0}.
Since $\curl(\Y-\Y^0)=\0$, from the first equality in \eqref{dec-0} we get \eqref{form-Y}.  Since $\div(\bold b+\h_2)=0$ in $\O$ and $\nu\times(\bold b+\h_2)=\nu\times\B^0$ on $\p\O$, from \eqref{MC-J4} we see that $\phi$ satisfies \eqref{eqdivB-MC-J}.
\end{proof}

\subsection{Reduce \eqref{eqdivB-MC-J} to Dirichlet problem \eqref{eqphi}}\

We assume that $\mH(x,\z)$ satisfies $(H_3)$ and the following condition:
\begin{itemize}
\item[$(H_{3T})$] For any $x\in\p\O$ and $\y, \z\in\Bbb R^3$, the equality $\z_T=\mB_T(x,\y)$ holds if and only if $\y_T=\mH_T(x,\z_T)$.
\end{itemize}
Here we use the notation
$$\z_T=\nu\times(\z\times\nu),\q \mH_T(x,\z)=\nu\times[\mH(x,\z)\times\nu],\q \mB_T(x,\w)=\nu\times[\mB(x,\w)\times\nu].
$$
We shall show that, BVP \eqref{eqdivB-MC-J} is equivalent to the following Dirichlet problem
\eq\label{eqphi}
\div [\mB(x,\Y^0+\h^0_1+\nabla\phi)]=0\q\text{\rm in }\O,\q
\phi=\phi_0\q\text{\rm on }\p\O,
\eeq
where $\phi_0$ is any function on $\p\O$ satisfying
\eq\label{dphi0}
(\nabla\phi_0)_T=\mH_T(x,\B^0_T)-\Y^0_T-\h^0_{1,T}\q\text{\rm on }\p\O.
\eeq
For a function $\phi$ defined on $\p\O$, we denote by $\nabla_T\phi$ the tangential gradient of $\phi$. If $\phi$ is extended over $\bar\O$, then
$$
\nabla_T\phi=(\nabla\phi)_T=(\nu\times\nabla\phi)\times\nu.
$$
Recall (see \cite[Lemma 2.5]{NW}) that, given a vector field $\w\in TC^1(\p\O,\Bbb R^3)$, a necessary and sufficient condition for existence of a function $\phi$ satisfying $\nabla_T\phi=\w$ is
\eq\label{A.2}
\nu\cdot\curl\w=0,\q \int_{\p\O} \langle\nu\times\w, \h_1\rangle dS=0\q\forall \h_1\in \Bbb H_1(\O).
\eeq

\begin{Prop}\label{Prop-eqdivB} Assume $\O, \J, \h_1$ satisfy \eqref{cond-MJ}, $\nu\times\B^0$ satisfies \eqref{condMC-J}, $\mH$ satisfies $(H_3)$ and  $(H_{3T})$.
\begin{itemize}
\item[(i)] Given $\Y^0$ satisfying \eqref{Y0} and $\h_1^0\in\Bbb H_1(\O)$, a necessary condition for solvability of \eqref{eqdivB-MC-J} is
\eq\label{condB3}
\aligned
&\nu\cdot\curl[\mH(x,\B^0_T)]=\nu\cdot\J\q\text{\rm on }\p\O,\\
&\int_{\p\O}\langle\nu\times[\mH_T(x,\B^0_T)-\Y^0-\h^0_{1}],\hat\h_1\rangle dS=0,\q\forall \hat\h_1\in\Bbb H_1(\O).
\endaligned
\eeq
\item[(ii)] If  $\B^0_T$ satisfies \eqref{condB3} and $\h^0_1\in\Bbb H_1(\O)$, then there exists a function $\phi_0$ on $\p\O$, which depends on $\h^0_1$,  such that \eqref{dphi0} holds, and BVP  \eqref{eqdivB-MC-J} is reduced to the Dirichlet BVP \eqref{eqphi}.

\item[(iii)] In addition to the assumption in (ii), if furthermore $\mH(x,\z)$ satisfies $(H_1)$ with $g_2\in L^2(\O)\cap L^2(\p\O)$, then
we can choose $\phi_0\in H^1(\p\O)$ such that
$$
\|\phi_0\|_{H^1(\p\O)}\leq C(\O)\{\|\mH_T(x,\B^0_T)\|_{L^2(\p\O)}+\|\J+\h_1\|_{L^2(\O)}+\|\h^0_1\|_{L^2(\O)}\}.
$$
Moreover, if $\p\O$ is of $C^{2,\a}$, $\J\in C^\a(\bar\O,\Bbb R^3)$, and if the right side of \eqref{dphi0} is of $C^{1,\a}$, then we can choose $\phi_0\in C^{2,\a}(\p\O)$ such that
$$
\|\phi_0\|_{C^{2,\a}(\p\O)}\leq C(\O,\a)\{\|\mH_T(x,\B^0_T)\|_{C^{1,\a}(\p\O)}+\|\J\|_{C^\a(\bar\O)}+
\|\h_1\|_{L^2(\O)}+\|\h^0_1\|_{L^2(\O)}\}.
$$
\end{itemize}
\end{Prop}

\begin{proof}
 (i). By $(H_{3T})$, boundary condition in \eqref{eqdivB-MC-J} is equivalent to
\eq\label{bd3}
(\nabla\phi)_T=\mH_T(x,\B^0_T)-\Y^0_T-\h^0_{1,T}\q\text{on }\p\O.
\eeq
\eqref{bd3} implies that the vector field $\mH_T(x,\B^0_T)-\Y^0_T-\h^0_{1,T}$ is a tangential component of a gradient, which is possible if and only if (see \eqref{A.2})
$$\aligned
&\nu\cdot\curl[\mH_T(x,\B^0_T)-\Y^0_T-\h^0_{1,T}]=0\q\text{on }\p\O,\\
&\int_{\p\O}\langle\nu\times[\mH_T(x,\B^0_T)-\Y^0_T-\h^0_{1,T}],\hat\h_1\rangle dS=0,\q\forall \hat\h_1\in\Bbb H_1(\O).
\endaligned
$$
This condition is equivalent to \eqref{condB3} because
$$\aligned
&\nu\cdot\curl\Y^0_T=\nu\cdot\curl\Y^0=\nu\cdot\J,\q \nu\cdot\curl\h^0_{1,T}=\nu\cdot\curl\h^0_1=0,\\
&\nu\cdot\curl[\mH_T(x,\B^0_T)]=\nu\cdot\curl[\mH(x,\B^0_T)].
\endaligned
$$

(ii). Assume $\B^0_T$ satisfies \eqref{condB3}. Then we can find $\phi_0\in H^1(\O)$ satisfying \eqref{dphi0}, such that
\eqref{bd3} can be written as $(\nabla\phi)_T=(\nabla\phi_0)_T$ on $\p\O$.
So we can reduce \eqref{eqdivB-MC-J} to \eqref{eqphi}.

(iii) From $(H_1)$ we have
$$
\|(\nabla\phi_0)_T\|_{L^2(\p\O)}\leq \|\mH_T(x,\B^0_T)\|_{L^2(\p\O)}+\|\Y^0_T\|_{L^2(\p\O)}+\|\h^0_1\|_{L^2(\p\O)}.
$$
This implies that $\phi_0\in H^1(\p\O)$. From \eqref{Y0} we have
$$
\|\Y^0_T\|_{L^2(\p\O)}\leq C(\O)\|\Y^0\|_{H^1(\O)}\leq C(\O)\|\J+\h_1\|_{L^2(\O)}.
$$
Using this and \eqref{H1H2} we have
$$
\|\phi_0\|_{H^1(\p\O)}\leq C\{\|\mH_T(x,\B^0_T)\|_{L^2(\p\O)}+\|\J+\h_1\|_{L^2(\O)}+\|\h^0_1\|_{L^2(\O)}\}.
$$

Next assume \eqref{condB3} holds, $\p\O$ is of $C^{2,\a}$, and the right side of \eqref{dphi0} is of $C^{1,\a}$. Then $(\nabla\phi_0)_T$ is of $C^{1,\a}$, so $\phi_0\in C^{2,\a}(\p\O)$. Since $\Y^0\in H^1_{n0}(\O,\div0)\cap \Bbb H_1(\O)^\perp_{L^2(\O)}$ and $\curl\Y^0=\J+\h_1\in C^\a(\bar\O,\Bbb R^3)$, by the Schauder regularity of the div-curl system we have $\Y^0\in C^{1,\a}(\bar\O,\Bbb R^3)$, and
$$
\|\Y^0\|_{C^{1,\a}(\bar\O)}\leq C(\O,\a)\|\J+\h_1\|_{C^\a(\bar\O)}\leq C(\O,\a)(\|\h_1\|_{L^2(\O)}+\|\J\|_{C^\a(\bar\O)}).
$$
Here we used the fact that $\h_1, \h^0_1\in \Bbb H_1(\O)$, hence
$$\|\h_1\|_{C^{1,\a}(\bar\O)}\leq C(\O,\a)\|\h_1\|_{L^2(\O)},\q \|\h^0_1\|_{C^{1,\a}(\bar\O)}\leq C(\O,\a)\|\h^0_1\|_{L^2(\O)}.
$$
Finally from \eqref{bd3} we have
$$\aligned
\|(\nabla\phi_0)_T\|_{C^{1,\a}(\p\O)}\leq& \|\mH_T(x,\B^0_T)\|_{C^{1,\a}(\p\O)}+C(\O,\a)\{\|\h_1\|_{L^2(\O)}+\|\h^0_1\|_{L^2(\O)}+\|\J\|_{C^\a(\bar\O)}\}.
\endaligned
$$
\end{proof}

\subsection{Solvability of \eqref{eqphi}}\

The Dirichlet problem \eqref{eqphi} has been very well studied.
We list the following existence result for \eqref{eqphi}, see for instance \cite[Chpater 4, Theorem 8.3]{LU} and \cite[Theorem 15.11]{GT}.

\begin{Lem}\label{Lem-Dphi} Assume $0<\a<1$, $\O$ is a bounded domain in $\Bbb R^3$ with a $C^{2,\a}$ boundary, $\mB\in C^{1,\a}_{\loc}(\bar\O\times\Bbb R^3,\Bbb R^3)$ satisfies $(B_1), (B_2)$, $\Y^0, \h^0_1\in C^{1,\a}(\bar\O,\Bbb R^3)$ and $\phi_0\in C^{2,\a}(\p\O)$. Then \eqref{eqphi} has a unique solution $\phi\in C^{2,\a}(\bar\O)$.
\end{Lem}
\begin{proof}
Given $\phi_0\in C^{2,\a}(\p\O)$, we can extend $\phi_0$ to a $C^{2,\a}$ function on $\bar\O$ and set $\psi=\phi-\phi_0$. Then $\phi$ solves \eqref{eqphi} if and only if $\psi$ solves
\eq\label{eqphi0}
\left\{\aligned
&\div [\mB(x,\Y^0+\h^0_1+\nabla\phi_0+\nabla\psi)]=0\q&\text{\rm in }\O,\\
&\psi=0\q&\text{\rm on }\p\O.
\endaligned\right.
\eeq
As in the proof of Proposition \ref{Prop-sol} (i), we define a map $\mT$ such that
$$
\langle \mT(\psi),\eta\rangle_{H^{-1}(\O), H^1_0(\O)}=\int_\O \langle\mB(x,\Y^0+\h_1^0+\nabla\phi_0+\nabla\psi), \nabla\eta\rangle dx,\q\forall \psi, \eta\in H^1_0(\O).
$$
Using conditions $(B_1)$ and $(B_2)$ we can show that $\mT: H^1_0(\O)\to H^{-1}(\O)$ is hemi-continuous and strongly monotone, hence it is surjective. So there exists a unique $\psi\in H^1_0(\O)$ such that
$\mT(\psi)=0,$
Thus $\psi$ is the unique solution of \eqref{eqphi0} in $H^1_0(\O)$.

Then using the conditions on $\mB$ and by the Schauder regularity theory of elliptic equations we see that $\psi\in C^{2,\a}(\bar\O)$.
\end{proof}

\subsection{Solvability of \eqref{MC-J}}\

Using Lemma \ref{Lem-Dphi} we have existence for \eqref{MC-J} under the following assumption
\eq\label{cond-for-MC-J}
\aligned
&\text{$\O$ is a bounded domain in $\Bbb R^3$ with a $C^{3,\a}$ boundary,}\q 0<\a<1,\\
&\text{$\mH\in C^{1,\a}_{\loc}(\bar\O\times\Bbb R^3,\Bbb R^3)$ satisfying $(H_1), (H_2), (H_3)$ and $(H_3T)$,}\\
&\J\in C^\a(\bar\O,\Bbb R^3)\cap \mH^\Gamma(\O,\div0), \q \h_1\in \Bbb H_1(\O),\q \nu\times\B^0\in T\!C^{1,\a}(\p\O,\Bbb R^3).
\endaligned
\eeq

\begin{Thm} Assume $\O, \mH, \nu\times\B^0, \J, \h_1$ satisfy \eqref{cond-for-MC-J}, $\Y^0$ is determined by $\J$ and $\h_1$ by \eqref{Y0}.
 For any $\h^0_1\in \Bbb  H_1(\O)$ such that \eqref{condB3} holds,  there exists an $\h_2^0\in \Bbb H_2(\O)$ which depends on $\h_1^0$, such that  \eqref{MC-J} with $\h_2=\h_2^0$ has a solution $\u\in C^{2,\a}(\bar\O,\Bbb R^3)$.

In particular, if $\O$ has no holes, and if there exists $\h^0_1\in \Bbb  H_1(\O)$ such that \eqref{condB3} holds, then \eqref{MC-J} has a solution.
\end{Thm}

\begin{proof} Since $\nu\times\B^0\in T\!C^{1,\a}(\p\O,\Bbb R^3)$, from Remark \ref{Rem-MCJ-1} (ii) we know that $\nu\times\B^0\in T_1[C^{1,\a}(\bar\O,\Bbb R^3)\cap \mH^\Gamma(\O,\div0)]$. Hence there exists $\bold b\in C^{1,\a}(\bar\O,\Bbb R^3)\cap \mH^\Gamma(\O,\div0)$ such that $\nu\times\bold b=\nu\times\B^0$ on $\p\O$. Then  $\bold b=\curl\U$ for some $\U\in H^1_{n0}(\O,\div0)\cap C^{2,\a}(\bar\O,\Bbb R^3)$, and we reduce \eqref{MC-J} to \eqref{MC-J1}.

Since $\p\O$ is of $C^{2,\a}$, so $\h_1, \h^0_1\in \Bbb H_1(\O)\subset C^{1,\a}(\bar\O,\Bbb R^3)\cap \mH^\Gamma(\O,\div0)$.
Since $\J\in  C^\a(\bar\O,\Bbb R^3)\cap\mH^\Gamma(\O,\div0)$, then $\J+\h_1\in C^{1,\a}(\bar\O,\Bbb R^3)\cap \mH^\Gamma(\O,\div0)$. So we can find a unique $\Y^0$ satisfying \eqref{Y0}, and by regularity of div-curl system we have $\Y^0\in C^{1,\a}(\bar\O,\Bbb R^3)$.

Since $\B^0_T, \Y^0, \h^0_1$ satisfy \eqref{condB3} and  $\mH(x,\z)$ satisfies $(H_{3T})$, from Proposition \ref{Prop-eqdivB}   we can find a function $\phi_0$ defined on $\p\O$ such that \eqref{dphi0} holds. Since $\mH$ is of $C^{1,\a}$, $\B^0_T$ and $\h^0_1$ are of $C^{1,\a}(\p\O,\Bbb R^3)$, so the right side of \eqref{dphi0} is of $C^{1,\a}$, hence $\phi_0$ can be chosen such that $\phi_0\in C^{2,\a}(\p\O)$.
By Lemma \ref{Lem-Dphi} the Dirichlet problem \eqref{eqphi} has a solution $\phi_{\h^0_1}\in C^{2,\a}(\bar\O)$, and by Proposition \ref{Prop-eqdivB}, $\phi_{\h^0_1}$ is a solution of \eqref{eqdivB-MC-J}.
Set
$$\Y_{\h_1^0}=\Y^0+\h_1^0+\nabla\phi_{\h_1^0}.
$$
By Lemma \ref{Lem-equiv3-MC-J} $\Y_{\h_1^0}$ is a $C^{1,\a}$ solution of \eqref{MC-J4}.

Now we show that, for the above given $\h^0_1$, we can find a unique $\h_2^0\in \Bbb H_2(\O)$ (and $\h_2^0$ depends on $\h^0_1$), such that
$\mB(x,\Y_{\h_1^0}(x))-\bold b(x)-\h_2^0(x)$
satisfies \eqref{orth6.9}.
To prove, let us write  the orthonormal basis of $\Bbb H_2(\O)$ as $\{\y^1,\cdots,\y^m\}$, and write the orthogonality condition \eqref{orth6.9} for $\B(x,\Y_{\h_1^0})$ as follows:
\eq\label{orth22}
\int_\O\langle \mB(x,\Y^0+\h^0_1+\nabla\phi_{\h^0_1}),\y^k\rangle dx=\int_\O(\bold b+\h_2^0)\cdot\y^k dx,\q k=1,\cdots,m.
\eeq
Obviously, for the given $\h^0_1\in\Bbb H_1(\O)$ we can find a unique $\h_2^0\in\Bbb H_2(\O)$ such that the above equalities hold. In fact we can choose
\eq\label{h20}
\h_2^0=\sum_{k=1}^m a^k\y^k,\q a^k=\int_\O\langle \mB(x,\Y^0+\h^0_1+\nabla\phi_{\h^0_1})-\bold b,\y^k\rangle dx.
\eeq
Then \eqref{orth6.9} holds.

For the above choice of $\h_2^0$, let
$$\Z_{\h_1^0,\h_2^0}=\mB(x,\Y_{\h_1^0})-\bold b(x)-\h^0_2.
$$
From Lemmas \ref{Lem-equiv1-MC-J}, \ref{Lem-equiv2-MC-J}, $\Z=\Z_{\h_1^0,\h_2}$ is a $C^{1,\a}$ solution of \eqref{MC-J2} and satisfies  \eqref{condZ}. Hence $\Z_{\h^0_1,\h^0_2}=\curl\w$ for some $\w\in H^1_{n0}(\O,\div0)$. Then $\w$ is a weak solution of \eqref{MC-J1} with $\h_2=\h_2^0$. Since $\Z_{\h^0_1,\h^0_2}\in C^{1,\a}(\bar\O,\Bbb R^3)$, we see that $\w\in C^{2,\a}(\bar\O,\Bbb R^3)$. Let $\u=\U+\w$. Then $\u\in C^{2,\a}(\bar\O,\Bbb R^3)$ is a solution of \eqref{MC-J} with $\h_2=\h^0_2$.
\end{proof}

\section{Other BVPs of the Maxwell System with a Given Current}

\subsection{The normal curl BVP of the Maxwell System}\

In this subsection we consider the following problem
\eq\label{MnC-J}
\left\{\aligned
&\curl [\mH(x,\curl\u+\h_2(x))]=\J(x)+\h_1(x)\q &\text{in }\O,\\
&\div\u=0\q &\text{in }\O,\\
&\nu\cdot\curl\u+\nu\cdot\h_2=B^0_n\q& \text{on }\p\O.
\endaligned\right.
\eeq
We assume \eqref{cond-MJ} and
\eq\label{condMnC-J}
B_n^0\in \dot H^{-1/2}(\p\O).
\eeq
It is natural to consider weak solutions of \eqref{MnC-J} in $\mH(\O,\curl,\div0)$.
A necessary condition for \eqref{MnC-J} to have a weak solution is  $\int_{\p\O} B_n^0 dS=0$ and $\J\in \mH^\Gamma(\O,\div0)$.

\begin{Lem} Assume $\J, \h_1$ satisfy \eqref{cond-MJ} and $B^0_n$ satisfies \eqref{condMnC-J}. Then \eqref{MnC-J} is solvable in $\mH(\O,\curl)$ for some $\h_2\in\Bbb H_2(\O)$ if and only if the following system has a solution  $\Z\in \mH(\O,\div0)$:
\eq\label{MnC-HZ}
\left\{
\aligned
&\curl\mH(x,\Z)=\J+\h_1,\qq\div\Z=0\q&\text{\rm in }\O,\\
&\nu\cdot\Z=B_n^0\q&\text{\rm on }\p\O.
\endaligned\right.
\eeq
\end{Lem}

\begin{proof}
If $\u\in \mH(\O,\curl)$ is a solution of \eqref{MnC-J} and let $\Z=\curl\u+\h_2$, then $\Z\in \mH(\O,\div0)$ is a solution of \eqref{MnC-HZ}.
On the other hand, if \eqref{MnC-HZ} has a solution $\Z\in \mH(\O,\div0)$, then we can always take $\h_2\in \Bbb H_2(\O)$ (which depends on $\Z$) such that $\Z-\h_2\perp_{L^2(\O)} \Bbb H_2(\O)$. Then $\Z-\h_2\in \mH^\Gamma(\O,\div0)$, so there exists $\u\in H^1_{n0}(\O,\div0)$ such that $\curl\u=\Z-\h_2$. Then  $\u$ is a solution of \eqref{MnC-J} corresponding to this $\h_2$.
\end{proof}

\begin{Lem}\label{Lem-equiv-MnC-J} Assume that $\mH$ satisfies $(H_1), (H_3)$, $\J, \h_1$ satisfy \eqref{cond-MJ} and $B^0_n$ satisfies \eqref{condMnC-J}. Let $\Y^0$ be determined by $\J$ and $\h_1$ by \eqref{Y0}.
\begin{itemize}
\item[(i)] If there exists $\h_2\in \Bbb H_2(\O)$ such that \eqref{MnC-J} has a solution $\u\in \mH(\O,\curl)$, then there exists $\h^0_1\in \Bbb H_1(\O)$ (depending on $\h_2$) such that the following problem
    \eq\label{eqdivB-MnC-J}
\left\{\aligned
&\div[\mB(x,\Y^0+\h^0_1+\nabla\phi)]=0\q&\text{\rm in }\O,\\
&\nu\cdot\mB(x,\Y^0+\h^0_1+\nabla\phi)=B^0_n\q&\text{\rm on }\p\O,
\endaligned\right.
\eeq
has a solution $\phi\in H^1(\O)$ which satisfies the orthogonality condition
\eq\label{orth-BYH2}
\mB(x,\Y^0+\h^0_1+\nabla\phi)-\h_2\perp_{L^2(\O)}\Bbb H_2(\O).
\eeq

\item[(ii)] If there exists $\h^0_1\in \Bbb H_1(\O)$  such that \eqref{eqdivB-MnC-J} has a solution $\phi$, then there exists $\h_2\in\Bbb H_2(\O)$ (depending on $\h_1^0$) such that  $(\phi,\h_2)$ satisfies \eqref{orth-BYH2}, and hence \eqref{MnC-J} has a solution $\u\in H^1_{n0}(\O,\div0)$.
    \end{itemize}
\end{Lem}

\begin{proof}
 $\Z$ is a solution of \eqref{MnC-HZ} if and only if  $\Y(x)=\mH(x,\Z(x))$ is a solution of
\eq\label{MnC-J-Y}
\left\{\aligned
&\curl\Y=\J+\h_1,\qq \div[\mB(x,\Y(x))]=0\q&\text{in }\O,\\
&\nu\cdot\mB(x,\Y(x))=0\q&\text{on }\p\O.
\endaligned\right.
\eeq
Similar to the proof of Lemma \ref{Lem-equiv3-MC-J} we can show that \eqref{MnC-J-Y} is equivalent to \eqref{eqdivB-MnC-J}.
\end{proof}

\begin{Thm}\label{Thm1-ex-MnC-J} Assume $\O, \mH(x,\z), \J, \h_1$ satisfy \eqref{cond-MJ} and $B^0_n$ satisfies \eqref{condMnC-J}. For any $\h_1^0\in\Bbb H_1(\O)$, there exists
$\h_2\in\Bbb H_2(\O)$ (depending $\h_1, \h_1^0$) such that \eqref{MnC-J} has a solution.
\end{Thm}

\begin{proof} Let $\Y^0$ be determined by $\J$ and $\h_1$ by \eqref{Y0}, and let $\h_1^0\in\Bbb H_1(\O)$.
As in the proof of Proposition \ref{Prop-sol} (i), we define a map $\mT$ by
$$
\langle \mT(\phi),\eta\rangle_{\dot H^1(\O)^*, \dot H^1(\O)}=\int_\O \langle\mB(x,\Y^0(x)+\h_1^0(x)+\nabla\phi), \nabla\eta\rangle dx,\q\forall\phi, \eta\in \dot H^1(\O).
$$
Conditions $(B_1)$, $(B_2)$ imply that $\mT$ is hemi-continuous and strongly monotone from $\dot H^1(\O)$ to $\dot H^1(\O)^*$, hence $\mT$ is surjective. So there exists a unique $\phi\in \dot H^1(\O)$ such that
$$
\langle \mT(\phi),\eta\rangle_{\dot H^1(\O)^*, \dot H^1(\O)}=\langle  B^0_n,\eta\rangle_{\p\O,1/2},\q\forall \eta\in \dot H^1(\O).
$$
$\phi$ is the unique weak  solution of \eqref{eqdivB-MnC-J} in $\dot H^1(\O)$. Note that $\phi$ depends on $\h_1,\h_1^0$. By choosing $\h_2\in\Bbb H_2(\O)$ depending on $\h_1, \h_1^0$ such that \eqref{orth-BYH2} holds, and using Lemma \ref{Lem-equiv-MnC-J} (ii), we conclude  that Eq. \eqref{MnC-J} with this choice of $\h_2$  has a weak solution $\u$.
\end{proof}

\subsection{The natural BVP of the Maxwell System}\

In this section we consider the natural BVP of the Maxwell system with a given current\footnote{Let us mention that in  \cite{MP1}, the electro-static problem is reduced to a problem of the form \eqref{MNa-J}.}:
\eq\label{MNa-J}
\left\{\aligned
&\curl [\mH(x,\curl\u+\h_2(x))]=\J(x)+\h_1(x)\q &\text{in }\O,\\
&\div\u=0\q &\text{in }\O,\\
&\nu\times\mH(x,\curl\u+\h_2(x))=\nu\times\H^0\q& \text{on }\p\O.
\endaligned\right.
\eeq
We assume \eqref{cond-MJ} and
\eq\label{condMNa-J1}
\nu\times\H^0\in T\!H^{1/2}(\p\O,\Bbb R^3),\q \nu\cdot\curl\H^0_T\in H^{-1/2}(\p\O).
\eeq

Let us denote by  $\pi_\tau$ the tangential trace map $\w\mapsto\w_T$ from $\mH(\O,\curl)$ to $H^{-1/2}(\p\O,\Bbb R^3)$, where $\w_T=(\nu\times\w)\times\nu$. In particular $\H^0_T=(\nu\times\H^0)\times\nu$.

\begin{Def} Assume $\O, \J, \h_1, \h_2$ satisfy \eqref{cond-MJ} and $\nu\times\H^0$ satisfies \eqref{condMNa-J1}.
$\u$ is called a weak solution of \eqref{MNa-J} if $\u\in \mH(\O,\curl,\div0)$,
 $\mH(x,\curl\u(x)+\h_2(x))\in L^{2,-1/2}_t(\O,\Bbb R^3)$, and for any $\w\in \mH(\O,\curl)$ it holds that
\eq\label{wk-MNa-J}
\int_\O (\J+\h_1)\cdot\w dx=\int_\O\mH(x,\curl\u+\h_2)\cdot\curl\w dx+\langle  \pi_\tau\w, \nu\times\H^0\rangle_{\p\O,1/2}.
\eeq
$\u$ is called an $H^1$-weak solution of \eqref{MNa-J} if $\u$ is a weak solution and $\u\in H^1(\O,\Bbb R^3)$.
\end{Def}

If $\u$ is a solution of \eqref{MNa-J}, then for any $\bold v\in \mH(\O,\curl0,\div0)$, $\u+\bold v$ is also a weak solution of \eqref{MNa-J}.
Therefore solvability of \eqref{MnC-J} in $\mH(\O,\curl,\div0)$ is equivalent to solvability of $H^1$-weak solutions in $H^1_{n0}(\O,\div0)\cap \Bbb H_1(\O)^\perp_{L^2(\O)}$.

\begin{Lem}\label{Lem-nec3} Assume $\O, \J, \h_1, \h_2$ satisfy \eqref{cond-MJ} and $\nu\times\H^0$ satisfies \eqref{condMNa-J1}.
If \eqref{MNa-J} has a weak solution $\u\in \mH(\O,\curl,\div0)$, then
\eq\label{condMNa-J}
\aligned
&\nu\cdot\curl\H^0_T=\nu\cdot\J\q\text{\rm on }\p\O,\\
&\int_\O(\J+\h_1)\cdot\h dx=\langle \pi_\tau\h, \nu\times\H^0\rangle_{\p\O,1/2},\q\forall \h\in \Bbb H_1(\O).
\endaligned
\eeq
\end{Lem}

\begin{proof} Assume \eqref{MNa-J} has a weak solution $\u\in H^1(\O,\Bbb R^3)$. Then
$$\aligned
&\nu\cdot\J=\nu\cdot\curl [\mH(x,\curl\u+\h_2)]=\nu\cdot\curl [\mH_T(x,\curl\u+\h_2)]
=\nu\cdot\curl\H^0_T.
\endaligned
$$
Taking $\w=\h\in \Bbb H_1(\O)$ in \eqref{wk-MNa-J} we get the last equality in \eqref{condMNa-J}.
\end{proof}

\begin{Rem} If $\J\in L^2(\O,\Bbb R^3)$, $\nu\times \H^0\in T\!H^{1/2}(\p\O,\Bbb R^3)$ and they satisfy the first equality in \eqref{condMNa-J}, then
\eq\label{eq-JH}
\int_\O\J\cdot\nabla\phi dx=\langle \pi_\tau(\nabla\phi), \nu\times\H^0\rangle_{\p\O,1/2},\q\forall \phi\in H^1(\O).
\eeq
\end{Rem}

\begin{proof} Let $\H^0$ be a divergence-free and tangential component preserving extension of $\H^0_T$ on to $\O$. From the first equality in \eqref{condMNa-J}, for any $\phi\in H^1(\O)$ we have
$$\aligned
&\int_\O\J\cdot\nabla\phi dx=\int_{\p\O}\phi\,\nu\cdot\J dS=\int_{\p\O}\phi\,\nu\cdot\curl\H^0_T dS
=\int_\O\nabla\phi\cdot\curl\H^0 dx\\
=&\int_\O\div(\H^0\times\nabla\phi)dx
=\langle \nu\cdot(\H^0\times\nabla\phi), 1\rangle_{\p\O,1/2}=\langle (\nabla\phi)_T, \nu\times\H^0\rangle_{\p\O,1/2}.
\endaligned
$$
\end{proof}

We have the following two existence results.

\begin{Prop} Assume $\O, \J, \h_1, \h_2, \nu\times\H^0$ satisfy \eqref{cond-MJ}, \eqref{condMNa-J1} and \eqref{condMNa-J}, and
$\mH(x,\z)=\nabla_\z P(x,\z)$
for a function $P$ satisfying $(P_1)$. Then  \eqref{MNa-J} has an $H^1$-weak solution $\u\in H^1_{n0}(\O,\div0)\cap\Bbb H_1(\O)^\perp_{L^2(\O)}$.
\end{Prop}
\begin{proof} Set $P_2(x,\z)=P(x,\z+\h_2(x))$.
Then
$\nabla_{\z}P_2(x,\z)=\mH(x,\z+\h_2(x))$, and
$P_2(x,\z)$ also satisfies the condition $(P_1)$ (with different constants and functions $g_1, g_2$).
Define
\eq\label{sp-Y}
Y=H^1_{n0}(\O,\div0)\cap \Bbb H_1(\O)^\perp_{L^2(\O)}=\mH^\Sigma_0(\O,\div0)\cap H^1(\O,\Bbb R^3),
\eeq
and
$$\mE[\w]=\int_\O\{P_2(x,\curl\w)-(\J+\h_1)\cdot\w\}dx+\langle\nu\times\H^0, \w_T\rangle_{\p\O,1/2}.
$$
By the div-curl-gradient inequalities we have
$$\|\w\|_{H^1(\O)}\leq C(\O)\|\curl\w\|_{L^2(\O)},\q\forall \w\in Y.
$$
As in the proof of Proposition 3.1 in \cite{P4} we can show that $\mE$ is coercive in $Y$, and $\mE$ has a minimizer $\w_0\in Y$.
From \eqref{condMNa-J} and \eqref{eq-JH} we have
$$
\int_\O(\J+\h_1)\cdot(\h+\nabla\phi) dx=\langle \pi_\tau(\h+\nabla\phi),\nu\times\H^0\rangle_{\p\O,1/2},\q\forall \h\in\Bbb H_1(\O),\q \phi\in H^1(\O).
$$
Hence
$\mE[\w+\h+\nabla\phi]=\mE[\w]$ for all $\w\in H^1_{n0}(\O,\div0)$, $\h\in \Bbb H_1(\O)$ and $\phi\in H^1(\O)$.
Recall that $\mH(\O,\curl)=H^1_{n0}(\O,\div0)\oplus \grad H^1(\O).$ Therefore the minimizer $\w_0$ of $\mE$ on $Y$ is also a minimizer of $\mE$ on $\mH(\O,\curl)$.
\end{proof}

\begin{Thm}\label{Thm-MNa-Mono} Assume $\O, \J, \h_1, \h_2, \nu\times\H^0$ satisfy \eqref{cond-MJ}, \eqref{condMNa-J1} and \eqref{condMNa-J}, and $\mH$ satisfies $(H_1), (H_2)$. Then  \eqref{MNa-J} has an $H^1$-weak solution $\u$,  and $\u$ is unique in $H^1_{n0}(\O,\div0)\cap \Bbb H_1(\O)^\perp_{L^2(\O)}$. The solution map $(\h_1,\J)\mapsto \u$ is continuous from $\Bbb H_1(\O)\times \mH^\Gamma(\O,\div0)$ to $H^1_{n0}(\O,\div0)$.
\end{Thm}

\begin{proof} The proof is similar to the proof of Theorem \ref{Thm-Mono-MD-J}.
Define $Y$ as in \eqref{sp-Y}.
For any given $\h_2\in \Bbb H_2(\O)$, we can define a map $\mA_{\h_2}: Y\to \Y^*$ such that
$$
A_{\h_2,\w}[\bv]=\int_\O \langle\mH(x,\curl\w+\h_2), \curl\bv\rangle dx,\q \forall \bv\in Y.
$$
We can show that $\mA_{\h_2}: Y\to Y^*$ is hemi-continuous and strongly monotone, hence surjective. We can also define an operator $\mL: L^2(\O,\Bbb R^3)\to Y^*$, such that
$$
\langle \mL(\u),\bv\rangle_{X^*,X}=\int_\O \u\cdot\bv dx-\langle\pi_\tau\bv,\nu\times\H^0\rangle_{\p\O,1/2},\q \forall \bv\in Y.
$$
Since $\mA_{\h_2}$ is surjective, there exists $\u\in Y$ such that
$\mA_{\h_2}(\u)=\mL(\J+\h_1),$
that is,
\eq\label{mAJ-2}
\int_\O \langle\mH(x,\curl\u+\h_2), \curl\bv\rangle dx=\int_\O(\J+\h_1)\cdot\bv dx-\langle\pi_\tau\bv,\nu\times\H^0\rangle_{\p\O,1/2},\q \forall \bv\in Y.
\eeq

For all $\h\in \Bbb H_1(\O)$ and $\phi\in H^1(\O)$ we have
$$\langle \mA_{\h_2}(\u),\h+\nabla\phi\rangle_{Y^*,Y}=0.
$$
On the other hand, from \eqref{condMNa-J} we have
$$\langle \mL(\J+\h_1),\h+\nabla\phi\rangle_{X^*,X}=\int_\O(\J+\h_1)\cdot(\h+\nabla\phi) dx-\langle\pi_\tau(\h+\nabla\phi),\nu\times\H^0\rangle_{\p\O,1/2}=0.
$$
From these and \eqref{mAJ-2} we see that \eqref{wk-MNa-J} holds for all $\w=\bv+\h+\nabla\phi$, where $\bv\in Y$, $\h\in \Bbb H_1(\O)$ and $\phi\in H^1(\O)$. So \eqref{wk-MNa-J} holds for all $\w\in \mH(\O,\curl)$.
Hence $\u\in Y$ is a weak solution of \eqref{MNa-J}.
By the strongly monotonicity of $\mA_{\h_2}$ we know that the weak solution $\u$ is unique in $Y$, and the map $\J+\h_1\mapsto \u$ is continuous from $Y^*$ to $Y$.
\end{proof}

\subsection{The co-normal BVP of the Maxwell System}\

In this section we consider the following
\eq\label{MCo-J}
\left\{\aligned
&\curl [\mH(x,\curl\u+\h_2(x))]=\J(x)+\h_1(x)\q &\text{in }\O,\\
&\div\u=0\q &\text{in }\O,\\
&\nu\cdot\mH(x,\curl\u+\h_2(x))=H^0_n\q&\text{on }\p\O.
\endaligned\right.
\eeq

\begin{Def} Assume $\O, \J, \h_1, \h_2$ satisfy \eqref{cond-MJ} and $H^0_n\in H^{-1/2}(\p\O)$. We say $\u$ is a weak solution of \eqref{MCo-J} if
\begin{itemize}
\item[(i)] $\u\in \mH(\O,\curl,\div0)$, $\mH(x,\curl\u(x)+\h_2(x))\in L^{2,-1/2}_\nu(\O,\Bbb R^3)$;
 \item[(ii)] the equality $\nu\cdot\mH(x,\curl\u(x)+\h_2(x))=H^0_n$ holds in the $H^{-1/2}(\p\O)$ sense;
\item[(iii)] for all $\bv\in H^1_{t0}(\O,\Bbb R^3)$ it holds that
$$\int_\O\langle\mH(x,\curl\u+\h_2),\curl\bv\rangle dx
=\int_\O\langle \J+\h_1,\bv\rangle dx.
$$
\end{itemize}
\end{Def}

To prove existence of solutions to \eqref{MCo-J},  in addition to \eqref{cond-MJ}, we assume the following conditions:
\begin{itemize}
\item[$(H_{3N})$] For any $x\in\p\O$ and $\z,\w\in\Bbb R^3$, the equality $\nu\cdot\z=\nu\cdot\mB(x,\w)$ holds if and only if $\nu\cdot\w=\nu\cdot\mH(x,\z)$.
\end{itemize}
\begin{itemize}
\item[($H^0$)] $H^0_n\in H^{1/2}(\p\O)$, and there exists $\H^0\in H^{1/2}(\p\O,\Bbb R^3)$ such that $\nu\cdot\H^0=H^0_n$.
\end{itemize}

\begin{Lem}\label{Lem-MCo} Assume that  $\O, \J, \h_1$ satisfy \eqref{cond-MJ}, $\mH$ satisfies $(H_1), (H_3), (H_{3N})$, $H^0_n$ satisfies  $(H^0$). Let $\Y^0$ be determined by $\J$ and $\h_1$ by \eqref{Y0}.
\begin{itemize}
\item[(i)] Given $\h_2\in\Bbb H_2(\O)$, if  \eqref{MCo-J} has a weak solution  $\u\in \mH(\O,\curl)$, then there exists $\h^0_1\in \Bbb H_1(\O)$ (depending on $\h_1, \h_2$) such that the following problem
  \eq\label{MCo-J2}
\left\{\aligned
&\div[\mB(x,\Y^0+\h^0_1+\nabla\phi)]=0\q &\text{\rm in }\O,\\
&\nu\cdot\mB(x,\Y^0+\h^0_1+\nabla\phi)=\nu\cdot\mB(x,\H^0)\q&\text{\rm on }\p\O,
\endaligned\right.
\eeq
 has a weak solution $\phi$, and $\phi$ satisfies orthogonality condition \eqref{orth-BYH2}.
\item[(ii)] If there exists $\h^0_1\in \Bbb H_1(\O)$  such that \eqref{MCo-J2} has a weak solution $\phi\in H^1(\O)$, and let $\h_2\in \Bbb H_2(\O)$ (depending on $\h_1, \h_1^0$) be such that  \eqref{orth-BYH2} holds, then  \eqref{MCo-J} has a weak solution $\u\in H^1_{n0}(\O,\div0)$.
\end{itemize}
\end{Lem}
\begin{proof} {\it Step 1}. If $\u\in \mH(\O,\curl)$ is a weak solution of \eqref{MCo-J}, let $\Z=\curl\u$. From the first equation in \eqref{MCo-J} we see that there exist $\h^0_1\in \Bbb H_1(\O)$ and $\phi\in H^1(\O)$ such that
$\mH(x,\Z+\h_2)=\Y^0+\h^0_1+\nabla\phi$.
By $(H_3)$ we have
$\Z=\mB(x,\Y^0+\h^0_1+\nabla\phi)-\h_2,
$
and from the second and third equalities in \eqref{MCo-J} we have
\eq\label{MCo-J3}
\left\{\aligned
&\div[\mB(x,\Y^0+\h^0_1+\nabla\phi)]=0\q &\text{in }\O,\\
&\nu\cdot\mH(x,\Z+\h_2)=H^0_n\q&\text{on }\p\O.
\endaligned\right.
\eeq
Since $\Z\in \curl \mH(\O,\curl,0)=\mH^\Gamma(\O,\div0)$, it satisfies the orthogonality condition \eqref{orth-BYH2}.
From $(H_{3N})$ and  $(H^0)$ we see that the boundary condition in \eqref{MCo-J3} is equivalent to the boundary condition in \eqref{MCo-J2}.

{\it Step 2}. Assume $(\h^0_1,\phi)$ solves \eqref{MCo-J2} and there exists $\h_2$ such that  \eqref{orth-BYH2} holds. Then $$\Z\equiv \mB(x,\Y^0(x)+\h^0_1(x)+\nabla\phi(x))-\h_2\in \mH^\Gamma(\O,\div0).
$$
Hence there exists $\u\in H^1_{n0}(\O,\div0)$ such that $\curl\u=\Z$. Then
$$
\curl\u+\h_2=\mB(x,\Y^0(x)+\h^0_1(x)+\nabla\phi(x)).
$$
By $(H_3)$, $(H_{3N})$, $(H^0)$ and \eqref{MCo-J2} we see that  $\u$ is a weak solution of \eqref{MCo-J} for this $\h_2$.
\end{proof}

\begin{Thm} Assume $\O, \J, \h_1, \mH(x,\z)$ satisfy \eqref{cond-MJ}, $(H_{3N})$, and $H^0_n$ satisfies $(H^0)$. For any $\h_1^0\in\Bbb H_1(\O)$, there exists
$\h_2\in\Bbb H_2(\O)$ (depending $\h_1, \h_1^0$) such that \eqref{MCo-J} has a weak solution $\u\in H^1_{n0}(\O,\div0)$.
\end{Thm}

\begin{proof} Given any $\h_1$, let $\Y^0$ be determined by $\J$ and $\h_1$ by \eqref{Y0}. Then $\Y^0\in H^1(\O,\Bbb R^3)$.
By using the monotone operator method (see the proofs of Proposition \ref{Prop-sol} (i), Lemma \ref{Lem-Dphi}),  we can show that for any $\h_1^0\in \Bbb H_1(\O)$, \eqref{MCo-J2} has a unique solution $\phi\in \dot H^1(\O)$
(and $\phi$ depends on $\h_1,\h_1^0$). We choose $\h_2\in\Bbb H_2(\O)$ which depends on $\h_1, \h_1^0$ such that \eqref{orth-BYH2} holds. Then from Lemma \ref{Lem-MCo}, for this choice of $\h_2$,  \eqref{MCo-J} has a weak solution $\u$.
\end{proof}

\section{The Maxwell-Stokes System with Solution-dependent Current}

In this section we consider the Maxwell-Stokes system where the current is of the form $\f(x,\curl\u+\h_2)$:
\eq\label{MSDN-f}
\left\{\aligned
&\curl [\mH(x,\curl\u+\h_2(x))]=\f(x,\curl\u+\h_2)+\h_1(x)+\nabla p\q &\text{in }\O,\\
&\div\u=0\q &\text{in }\O.\\
\endaligned\right.
\eeq
Throughout this section we assume that
\eq\label{cond-MSf}
\aligned
&\text{$\O$ is a bounded domain in $\Bbb R^3$ with a $C^2$ boundary,}\\
&\text{$\mH(x,\z)$ satisfies $(H_1), (H_2), (H_3)$,\q  $\f(x,\z)$ satisfies $(f_1)$,}\\
&\h_1\in\Bbb H_1(\O),\q\h_2\in\Bbb H_2(\O).
\endaligned
\eeq

It was observed in \cite{P4} that, if \eqref{MSDN-f} has a solution $\u$, then
$\f(x,\curl\u+\h_2)+\h_1+\nabla p\in \mH^\Gamma(\O,\div0),$
which holds if and only if
$$
\aligned
&-\Delta p=\div[\f(x,\curl\u(x)+\h_2(x))]\q\text{in }\O,\\
&\int_{\Gamma_i}\{{\p p\over\p\nu}+\nu\cdot\f(x,\curl\u(x)+\h_2(x))\}dS=0,\q i=1,\cdots, m+1.
\endaligned
$$
It suggests $p$ satisfies  Neumann type boundary condition. In this section, in addition to a boundary condition for $\u$.\footnote{The case where $\u$ satisfies a Dirichlet boundary condition $\u_T=\u^0_T$ can be treated by the methods in \cite{P4}, hence will not be discussed here. We also mention that most of analyzes in this section are valid for the nonlinear term of the form $\f(x,\u+\h_2)$.} we assume $p$ satisfies the following boundary condition
$${\p p\over\p\nu}=-\nu\cdot\f(x,\curl\u(x)+\h_2(x))\q \text{on }\p\O,
$$

\subsection{The normal curl BVP}\

In this subsection we consider the following problem
\eq\label{MSnCN-f1}
\left\{\aligned
&\curl [\mH(x,\curl\u+\h_2(x))]=\f(x,\curl\u+\h_2(x))+\h_1(x)+\nabla p\q &\text{in }\O,\\
&\div\u=0\q &\text{in }\O,\\
&\nu\cdot\curl\u+\nu\cdot\h_2=B^0_n,\q {\p p\over\p\nu}=-\nu\cdot\f(x,\curl\u+\h_2)\q &\text{on }\p\O.
\endaligned\right.
\eeq
We assume \eqref{cond-MSf} and $B^0_n$ satisfies \eqref{condMnC-J}.

\begin{Def}\label{Def-MSnCN-f1} We say that $(\u,p)$ is a weak solution of \eqref{MSnCN-f1} if $\u\in \mH(\O,\curl,\div0)$ and $p\in H^1(\O)$ such that
\begin{itemize}
\item[(i)] $\mH(x,\curl\u(x)+\h_2(x))\in L^2(\O,\Bbb R^3)$, $\f(x,\curl\u(x)+\h_2(x))\in L^{2,-1/2}(\O,\Bbb R^3)$;
\item[(ii)] the equality $\nu\cdot\curl\u(x)+\nu\cdot\h_2(x)=B^0_n$ holds in $H^{-1/2}(\p\O)$;
\item[(iii)]
\eq\label{wk-MSnCN-f1}
\aligned
&\int_\O\langle\mH(x,\curl\u+\h_2),\curl\bv\rangle dx\\
=&\int_\O\langle \f(x,\curl\u+\h_2)+\h_1+\nabla p,\bv\rangle dx, \q \forall \bv\in H^1_{t0}(\O,\Bbb R^3),
\endaligned
\eeq
\eq\label{wk-MSnCN-f1p}
\int_\O\langle \f(x,\curl\u+\h_2)+\nabla p,\nabla\eta\rangle dx=0,\q\forall \eta\in H^1(\O).
\eeq
\end{itemize}
\end{Def}

We shall show the relation between \eqref{MSnCN-f1} with the $Z$-system
\eq\label{MSnCN-f1Z}
\left\{\aligned
&\curl [\mH(x,\Z+\h_2)]=\mP_\nu[\f(x,\Z+\h_2)]+\h_1\q &\text{in }\O,\\
&\div\Z=0\q &\text{in }\O,\\
&\nu\cdot\Z+\nu\cdot\h_2=B^0_n\q& \text{on }\p\O,
\endaligned\right.
\eeq
and we shall show the equivalence of $Z$-system \eqref{MSnCN-f1Z} and the $Y$-system
\eq\label{MSnCN-f1Y}
\left\{\aligned
&\curl\Y=\mP_\nu[\f(x,\mB(x,\Y))]+\h_1\q &\text{in }\O,\\
&\div\mB(x,\Y)=0\q &\text{in }\O,\\
&\nu\cdot\mB(x,\Y)=B^0_n\q& \text{on }\p\O.
\endaligned\right.
\eeq
Here $\mP_\nu$ is the projection operator onto $\mH^\Gamma(\O,\div0)$, see section \ref{SecA}.

\begin{Def}\label{Def-MSnCN-f1Z} {\rm (i)} We say $\Z$ is a weak solution of \eqref{MSnCN-f1Z} if $\Z\in \mH(\O,\div0)$ such that
$$\aligned
&\mH(x,\Z(x)+\h_2(x))\in L^2(\O,\Bbb R^3),\q \f(x,\Z(x)+\h_2(x))\in L^2(\O,\Bbb R^3),\\
&\nu\cdot\Z+\nu\cdot\h_2=B^0_n\q\text{holds in the sense of $H^{-1/2}(\p\O)$,}\\
&\int_\O\langle \mH(z,\Z+\h_2), \curl\bv\rangle dx
=\int_\O \langle \mP_\nu[\f(x,\Z+\h_2)]+\h_1,\bv\rangle dx,\q\forall \bv\in H^1_{t0}(\O,\Bbb R^3).
\endaligned
$$

{\rm (ii)} We say $\Y$ is a weak solution of \eqref{MSnCN-f1Y} if $\Y\in L^2(\O,\Bbb R^3)$ such that
$$\aligned
&\mB(x,\Y(x))\in \mH(\O,\div0),\q \f(x,\mB(x,\Y(x)))\in L^2(\O,\Bbb R^3),\\
&\nu\cdot\mB(x,\Y(x))=B^0_n\q\text{holds in the sense of $H^{-1/2}(\p\O)$,}\\
&\int_\O\langle \Y, \curl\bv\rangle dx=\int_\O \langle \mP_\nu[\f(x,\mB(x,\Y))]+\h_1,\bv\rangle dx,\q\forall \bv\in H^1_{t0}(\O,\Bbb R^3).
\endaligned
$$
\end{Def}

\begin{Lem}\label{Lem-MSnCN-f1Y}  Assume \eqref{cond-MSf} and \eqref{condMnC-J}.
\begin{itemize}
\item[(i)] If $\Y\in L^2(\O,\Bbb R^3)$ is a weak solution of \eqref{MSnCN-f1Y}, then $\Y\in \mH(\O,\curl)$, and the first equality in \eqref{MSnCN-f1Y} holds a.e. in $\O$.

\item[(ii)] If $(\u,p)\in \mH(\O,\curl,\div0)\times H^1(\O)$ is a weak solution of \eqref{MSnCN-f1} and letting $\Z=\curl\u$, then $\Z\in \mH(\O,\div0)$ and it is a weak solution of \eqref{MSnCN-f1Z}. On the other hand, if $\Z\in \mH(\O,\div0)$ is a weak solution of \eqref{MSnCN-f1Z}, then there exists $\h_2^0\in \Bbb H_2(\O)$ such that \eqref{MSnCN-f1} with $\h_2$ replaced by $\h_2+\h_2^0$ has a weak solution $(\u,p)\in H^1_{n0}(\O,\div0)\times H^1(\O)$, and $\curl\u=\Z-\h_2^0$. If furthermore $\Z\perp_{L^2(\O)}\Bbb H_2(\O)$, then $\h^0_2=\0$.
\item[(iii)] If $\Z\in \mH(\O,\div0)$ is a weak solution of \eqref{MSnCN-f1Z} and letting
$\Y=\mH(x,\Z(x)+\h_2(x)),$
then $\Y\in \mH(\O,\curl)$ and it is a weak solution of \eqref{MSnCN-f1Y}. On the other hand, if $\Y\in L^2(\O,\Bbb R^3)$ is a weak solution of \eqref{MSnCN-f1Y} and letting
$\Z=\mB(x,\Y(x))-\h_2(x),$
then $\Z\in \mH(\O,\div0)$ and it is a weak solution of \eqref{MSnCN-f1Z}.
\end{itemize}
\end{Lem}
\begin{proof} We only prove (i). Assume $\Y\in L^2(\O,\Bbb R^3)$ is a weak solution of \eqref{MSnCN-f1Y}. By Definition \ref{Def-MSnCN-f1Z} (ii) and by $(f_1)$ we see that, for all $\bv\in H^1_{t0}(\O,\Bbb R^3)$ we have
$$\aligned
\int_\O\langle\Y,\curl\bv\rangle dx=&\int_\O\langle\mP_\nu[\f(x,\mB(x,\Y))]+\h_1,\bv\rangle dx\\
\leq & \|\mP_\nu[ \f(x,\mB(x,\Y))]+\h_1\|_{L^2(\O)}\|\bv\|_{L^2(\O)}.
\endaligned
$$
So $\curl\Y$ exists as a functional on $H^1_{t0}(\O,\Bbb R^3)$. For any $\w\in \mathcal D(\O,\Bbb R^3)$ we have
$$\aligned
&\langle\curl\Y,\w\rangle_{\mathcal D(\O,\Bbb R^3), \mathcal D'(\O,\Bbb R^3)}=\int_\O\langle\Y,\curl\w\rangle dx\\
=&\int_\O\langle\mP_\nu[\f(x,\mB(x,\Y))]+\h_1,\w\rangle dx
\leq \|\mP_\nu[ \f(x,\mB(x,\Y))]+\h_1\|_{L^2(\O)}\|\w\|_{L^2(\O)}.
\endaligned
$$
Hence $\curl\Y\in L^2(\O,\Bbb R^3)$,   and
$\curl\Y=\mP_\nu[\f(x,\mB(x,\Y))]+\h_1.$
\end{proof}

To solve \eqref{MSnCN-f1Y}, we use \eqref{dec-1} to write $\Y=\w+\hat\h_1+\nabla\psi$, with $\w\in \mH^\Sigma_0(\O,\div0)$, $\hat\h_1\in\Bbb H_1(\O)$, $\psi\in H^1(\O)$.
If $\Y$ satisfies \eqref{MSnCN-f1Y} then $\psi$ satisfies
\eq\label{MSnCN-f1B}
\left\{\aligned
&\div\mB(x,\w+\hat\h_1+\nabla\psi)=0\q &\text{in }\O,\\
&\nu\cdot\mB(x,\w+\hat\h_1+\nabla\psi)=B^0_n\q&\text{on }\p\O.
\endaligned\right.
\eeq

\begin{Prop}\label{Prop-ex-MnC-BY} Assume $\O, \mH(x,\z), B_n^0$ satisfy \eqref{cond-MSf} and \eqref{condMnC-J}, $\w\in \mH^\Sigma_0(\O,\div0)$ and $\hat\h_1\in\Bbb H_1(\O)$.
Then  \eqref{MSnCN-f1B} has a unique solution $\psi=\psi_{\w,\hat\h_1}\in \dot H^1(\O)$.
 The map $(\w,\hat\h_1)\mapsto \nabla\psi_{\w,\hat\h_1}$ is a locally Lipschitz and bounded map from $\mH^\Sigma_0(\O,\div0)\times\Bbb H_1(\O)$ into $\text{grad} H^1(\O)$.
\end{Prop}

\begin{proof} As in the proof of Proposition \ref{Prop-sol} (i), we can define a map $\mT$ such that
$$
\langle \mT(\psi),\eta\rangle_{\dot H^1(\O)^*, \dot H^1(\O)}=\int_\O \langle\mB(x,\w(x)+\hat\h_1(x)+\nabla\psi), \nabla\eta\rangle dx,\q\forall \psi, \eta\in \dot H^1(\O).
$$
Using $(B_1)$ and $(B_2)$ we can show $\mT: \dot H^1(\O)\mapsto \dot H^1(\O)^*$ is Lipschitz and strongly monotone, hence $\mT$ is surjective. Hence there exists a unique $\psi\in \dot H^1(\O)$ such that
$$
\langle \mT(\psi),\eta\rangle_{\dot H^1(\O)^*, \dot H^1(\O)}=\langle  B^0_n,\eta\rangle_{\p\O,1/2},\q\forall \eta\in \dot H^1(\O).
$$
From this and $(B_1)$ we have
$$\aligned
&{c_3\over 2}\int_\O|\w+\hat\h_1+\nabla\psi_{\w,\hat\h_1}|^2-{1\over 2}\|g_3\|_{L^2(\O)}^2\\
\leq& \int_\O\langle \mB(x,\w+\hat\h_1+\nabla\psi_{\w,\hat\h_1}),\w+\hat\h_1+\nabla\psi_{\w,\hat\h_1}\rangle dx\\
=&\int_\O\langle \mB(x,\w+\hat\h_1+\nabla\psi_{\w,\hat\h_1}),\w+\hat\h_1\rangle dx+\langle B_n^0,\psi_{\w,\hat\h_1}\rangle_{\p\O,1/2}\\
\leq& {c_3\over 4}\int_\O|\w+\hat\h_1+\nabla\psi_{\w,\hat\h_1}|^2dx+({4c_4^2\over c_3}+{c_4^2\over 2})\int_\O|\w+\hat\h_1|^2dx+{1\over 2}\|g_4\|_{L^2(\O)}^2\\
&+{c_3\over 8}\|\nabla\psi_{\w,\hat\h_1}\|_{L^2(\O)}+C\|B^0_n\|_{H^{-1/2}(\p\O)},
\endaligned
$$
here we have used the Poincar\'e inequality for $\psi_{\w,\hat\h_1}\in \dot H^1(\O)$.
Hence
\eq\label{H1est-MSnCN-f1B}
\|\nabla\psi_{\w,\hat\h_1}\|_{L^2(\O)}\leq C(\|\w+\hat\h_1\|_{L^2(\O)}+\|B^0_n\|_{H^{-1/2}(\p\O)}+1),
\eeq
where $C$ depends only on the constants and $g_3, g_4$ in $(B_1)$. Thus $(\w,\hat\h_1)\mapsto \nabla\psi_{\w,\hat\h_1}$ is a bounded map with respect to the $L^2$ norms.

Using $(B_2)$ and Remark \ref{Lem-mono-B} we have
$$\aligned
0=&\int_\O\langle \mB(x,\w+\hat\h_1+\nabla\psi_{\w,\hat\h_1})-\mB(x,\w'+\hat\h_1'+\nabla\psi_{\w',\hat\h_1'}), \nabla(\psi_{\w,\hat\h_1}-\psi_{\w,\hat\h_1'})\rangle dx\\
=&\int_\O\big\langle \mB(x,\w+\hat\h_1+\nabla\psi_{\w,\hat\h_1})-\mB(x,\w'+\hat\h_1'+\nabla\psi_{\w',\hat\h_1'}), \\
&\qq \w+\hat\h_1+\nabla\psi_{\w,\hat\h_1}-\w'-\hat\h_1'-\nabla\psi_{\w',\hat\h_1'}\rangle dx\\
&-\int_\O \langle \mB(x,\w+\hat\h_1+\nabla\psi_{\w,\hat\h_1})-\mB(x,\w'+\hat\h_1'+\nabla\psi_{\w',\hat\h_1'}), \w+\hat\h_1-\w'-\hat\h_1'\big\rangle dx\\
\geq &\lam_0\|\w+\hat\h_1+\nabla\psi_{\w,\hat\h_1}-\w'-\hat\h_1'-\nabla\psi_{\w',\hat\h_1'}\|_{L^2(\O)}^2\\
&-\int_\O\langle \mB(x,\w+\hat\h_1+\nabla\psi_{\w,\hat\h_1})-\mB(x,\w'+\hat\h_1'+\nabla\psi_{\w',\hat\h_1'}), \w+\hat\h_1-\w'-\hat\h_1'\rangle dx.
\endaligned
$$
From this and using $(B_1)$ we have
$$\aligned
&\lam_0\|\w+\hat\h_1+\nabla\psi_{\w,\hat\h_1}-\w'-\hat\h_1'-\nabla\psi_{\w',\hat\h_1'}\|_{L^2(\O)}^2\\
\leq &\int_\O\langle \mB(x,\w+\hat\h_1+\nabla\psi_{\w,\hat\h_1})-\mB(x,\w'+\hat\h_1'+\nabla\psi_{\w',\hat\h_1'}), \w+\hat\h_1-\w'-\hat\h_1'\rangle dx\\
\leq & C\{\|\w+\hat\h_1+\nabla\psi_{\w,\hat\h_1}\|_{L^2(\O)}+\|\w'+\hat\h_1'+\nabla\psi_{\w',\hat\h_1'}\|_{L^2(\O)}+\|g_4\|_{L^2(\O)}\}
\|\w+\hat\h_1-\w'-\hat\h_1'\|_{L^2(\O)}\\
\leq &C\{\|\w+\hat\h_1\|_{L^2(\O)}+\|\w'+\hat\h_1'\|_{L^2(\O)}+1\}\|\w+\hat\h_1-\w'-\hat\h_1'\|_{L^2(\O)}.
\endaligned
$$
Here we have used \eqref{H1est-MSnCN-f1B}. It follows that
$$
\|\nabla\psi_{\w,\hat\h_1}-\nabla\psi_{\w',\hat\h_1'}\|_{L^2(\O)}\leq C\|\w+\hat\h_1-\w'-\hat\h_1'\|_{L^2(\O)},
$$
where the constant $C$ is linearly increases in $\|\w+\hat\h_1\|_{L^2(\O)}+\|\w'+\hat\h_1'\|_{L^2(\O)}$. It follows that
the map $(\w,\hat\h_1)\mapsto \nabla\psi_{\w,\hat\h_1}$ is locally Lipschitz in $L^2(\O,\Bbb R^3)$.
\end{proof}

For any $\bold w\in \mH^\Sigma_0(\O,\div0)$, $\hat\h_1\in\Bbb H_1(\O)$, we denote
$$
\Y_{\w,\hat\h_1}=\w+\hat\h_1+\nabla\psi_{\w,\hat\h_1},\q
\Z_{\w,\hat\h_1}= \mB(x,\w+\hat\h_1+\nabla\psi_{\w,\hat\h_1})-\h_2.
$$
We look for $\bold v\in \mH^\Sigma_0(\O,\div0)$  such that
\eq\label{MSnCN-f1v}
\left\{\aligned
&\curl \bv=\mP_\nu[\f(x,\mB(x,\Y_{\w,\hat\h_1}))]+\h_1\q &\text{in }\O,\\
&\div\bv=0\q&\text{in }\O,\\
&\nu\cdot\bv=0\q&\text{on }\p\O.
\endaligned\right.
\eeq

\begin{Lem}\label{Lem-MSnCN-f1v} Assume \eqref{cond-MSf} and \eqref{condMnC-J}. Let $\hat\h_1\in\Bbb H_1(\O)$.
\begin{itemize}
\item[(i)] For any $\w\in \mH^\Sigma_0(\O,\div0)$ and  \eqref{MSnCN-f1v} has a unique solution $\bv=\bold V_{\hat\h_1}[\w]\in H^1(\O,\Bbb R^3)\cap \mH^\Sigma_0(\O,\div0)$.
\eq\label{est-MSnCN-f1v}
\|\bold V_{\hat\h_1}[\w]\|_{H^1(\O)}\leq C\{\|\w+\hat\h_1\|_{L^2(\O)}+\|B^0_n\|_{H^{-1/2}(\p\O)}\}+C,
\eeq
where $C$ depends on $\O, \mH(x,\z)$ and $\f(x,\z)$. The map $\bold V_{\hat\h_1}$ is a compact map in $\mH^\Sigma_0(\O,\div0)$.
\item[(ii)] If in addition $\f(x,\z)$ satisfies
\eq\label{f0}
\lim_{|\z|\to\infty} {|\f(x,\z)|\over|\z|}=0\q\text{uniformly for $x\in\bar\O$},
\eeq
then $\bold V_{\hat\h_1}$ has a fixed point $\w_{\hat\h_1}$.
\end{itemize}
\end{Lem}

\begin{proof} By the definition of $\mP_\nu$,  we know that the right side of the first equation of \eqref{MSnCN-f1v} lies in $\mH^\Gamma(\O,\div0)$. Hence \eqref{MSnCN-f1v} has a unique solution $\bv\in H^1_{n0}(\O,\div0)\cap \Bbb H_1(\O)^\perp_{L^2(\O)}= H_1(\O,\Bbb R^3)\cap \mH^\Sigma_0(\O,\div0)$. We have
$$\aligned
\|\bv\|_{H^1(\O)}\leq& C(\O)\|\mP_\nu[\f(x,\mB(x,\Y_{\w,\hat\h_1}))]+\h_1\|_{L^2(\O)}.
\endaligned
$$
Using $(B_1)$, \eqref{f0}, \eqref{H1est-MSnCN-f1B}, and Lemma \ref{Lem-A.4}, we get \eqref{est-MSnCN-f1v}.
Hence $\bold V_{\hat\h_1}$ is continuous and bounded from $\mH^\Sigma_0(\O,\div0)$ into $H^1(\O,\Bbb R^3)\cap\mH^\Sigma_0(\O,\div0)$. Then by the Sobolev imbedding theorem we know that $\bold V_{\hat\h_1}$ is compact from $\mH^\Sigma_0(\O,\div0)$ into itself.

If $\f(x,\z)$ satisfies \eqref{f0}, then there exists $R>0$ such that $\bold V_{\hat\h_1}$ maps the closed ball $B_R$ with center $\0$ and  radius $R$ in $\mH^\Sigma_0(\O,\div0)$ into $B_R$ itself. By Schauder's fixed point theorem we conclude that $\bold V_{\hat\h_1}$ has a fixed point in $B_R$.
\end{proof}

\begin{Thm}\label{Thm-MANaN-f1} Assume \eqref{cond-MSf}, \eqref{condMnC-J} and \eqref{f0}. Then for any $\hat\h_1\in\Bbb H_1(\O)$ there exists $\h_2=\h_2(\h_1,\hat\h_1)$ such that  \eqref{MSnCN-f1} with $\h_2=\h_2(\h_1,\hat\h_1)$ has a weak solution $(\u,p)\in H^1(\O,\Bbb R^3)\times H^1(\O)$.
\end{Thm}

\begin{proof} Let $\w_{\hat\h_1}$ denote the fixed point of the map $\bold V_{\hat\h_1}$ in $\mH^\Sigma_0(\O,\div0)$, and let
$$\aligned
&\psi_{\hat\h_1}=\psi_{\w_{\hat\h_1},\hat\h_1},\q
\Y_{\hat\h_1}=\w_{\hat\h_1}+\hat\h_1+\nabla\psi_{\hat\h_1},\q
\Z_{\hat\h_1}= \mB(x,\w_{\hat\h_1}+\hat\h_1+\nabla\psi_{\hat\h_1})-\h_2.
\endaligned
$$
$\w_{\hat\h_1}$ satisfies
\eq\label{MScNN-f1wh1}
\left\{\aligned
&\curl \w_{\hat\h_1}=\mP_\nu[\f(x,\mB(x,\Y_{\hat\h_1}))]+\h_1\q &\text{in }\O,\\
&\div\w_{\hat\h_1}=0\q&\text{in }\O,\\
&\nu\cdot\w_{\hat\h_1}=0\q&\text{on }\p\O.
\endaligned\right.
\eeq
Then $\Y_{\hat\h_1}$ is a solution of \eqref{MSnCN-f1Y}. By Lemma \ref{Lem-MSnCN-f1Y} (iii), $\Z_{\hat\h_1}$ is a solution of \eqref{MSnCN-f1Z}.
Then there exists $\h_2=\h_2(\h_1,\hat\h_1)$ depending on $\h_1$ and $\hat\h_1$, such that $\Z_{\hat\h_1}\perp_{L^2(\O)}\Bbb H_2(\O)$, namely
\eq\label{MSnCN-f1-orth}
\int_\O\langle \mB(x,\w_{\hat\h_1}+\hat\h_1+\nabla\psi_{\hat\h_1})-\h_2(\h_1,\hat\h_1), \h\rangle dx=0,\q\forall \h\in \Bbb H_2(\O).
\eeq
Then $\Z_{\hat\h_1}\in \mH^\Gamma(\O,\div0)$.  By Lemma \ref{Lem-MSnCN-f1Y} (ii), there exists $(\u,p)\in H_{n0}^1(\O,\div0)\times H^1(\O)$ which solves \eqref{MSnCN-f1} for the given $\h_1$.
\end{proof}

\subsection{The natural-Neumann BVP}\

In this section we consider the natural BVP of the Maxwell-Stokes system
\eq\label{MSNaN-f1}
\left\{\aligned
&\curl [\mH(x,\curl\u+\h_2)]=\f(x,\curl\u+\h_2)+\h_1(x)+\nabla p\q &\text{in }\O,\\
&\div\u=0\q &\text{in }\O,\\
&\nu\times\mH(x,\curl\u+\h_2)=\nu\times\H^0,\q {\p p\over\p\nu}=\nu\cdot\curl\H^0_T-\nu\cdot\f(x,\curl\u+\h_2)\q& \text{on }\p\O,
\endaligned\right.
\eeq
 We assume \eqref{cond-MSf} and
\eq\label{cond-MSNaN-f1}
\nu\times\H^0\in H^{-1/2}(\p\O,\Bbb R^3),\q\nu\cdot\curl\H^0_T\in H^{-1/2}(\p\O),
\eeq
where $\H^0_T=(\nu\times\H^0)\times\nu$.

\begin{Def}\label{Def-MSNaN-f1} We say that $(\u,p)$ is a weak solution of \eqref{MSNaN-f1} if $\u\in \mH(\O,\curl,\div0)$, $p\in H^1(\O)$, such that
\begin{itemize}
\item[(i)] $\mH(x,\curl\u(x)+\h_2(x))\in  L^{2,-1/2}_t(\O,\Bbb R^3)$,
$\f(x,\curl\u(x)+\h_2(x))\in L^2(\O,\Bbb R^3)$;
\item[(ii)]

\eq\label{wk-MSNaN-f1}
\aligned
&\int_\O\langle\mH(x,\curl\u+\h_2),\curl\bv\rangle dx+\langle \nu\times\H^0,\bv\rangle_{\p\O,1/2}\\
=&\int_\O\langle \f(x,\curl\u+\h_2)+\h_1+\nabla p,\bv\rangle dx,\q \forall \bv\in H^1(\O,\Bbb R^3),
\endaligned
\eeq
\eq\label{wk-MSNaN-f1p}
\int_\O\langle \f(x,\curl\u+\h_2)+\nabla p,\nabla\eta\rangle dx=\langle\nu\cdot\curl\H^0_T,\eta\rangle_{\p\O, 1/2},\q\forall \eta\in H^1(\O).
\eeq
\end{itemize}
\end{Def}

Regarding \eqref{MSNaN-f1} we have the following  $Z$-system
\eq\label{MSNaN-f1Z}
\left\{\aligned
&\curl [\mH(x,\Z+\h_2)]=\mP_n[\f(x,\Z+\h_2)]+\h_1\q &\text{in }\O,\\
&\div\Z=0\q &\text{in }\O,\\
&\nu\times\mH(x,\Z+\h_2)=\nu\times\H^0\q& \text{on }\p\O,
\endaligned\right.
\eeq
and $Y$-system
\eq\label{MSNaN-f1Y}
\left\{\aligned
&\curl\Y=\mP_n[\f(x,\mB(x,\Y))]+\h_1\q &\text{in }\O,\\
&\div\mB(x,\Y)=0\q &\text{in }\O,\\
&\nu\times\Y=\nu\times\H^0\q& \text{on }\p\O.
\endaligned\right.
\eeq
where $\mP_n$ is the Neumann projection associated with $\nu\cdot\curl\H^0_T$, see section \ref{SecA}.

\begin{Def}\label{Def-MSNaN-f1Z}{\rm (i)}
We say $\Z$ is a weak solution of \eqref{MSNaN-f1Z} if $\Z\in \mH(\O,\div0)$ such that
$$
\mH(x,\Z(x)+\h_2(x))\in L^{2,-1/2}_t(\O,\Bbb R^3),\q \f(x,\Z(x)+\h_2(x))\in L^{2,-1/2}_\nu(\O,\Bbb R^3),
$$
and for all $\bv\in H^1(\O,\Bbb R^3)$ it holds that
$$
\int_\O\langle \mH(z,\Z+\h_2), \curl\bv\rangle dx+\langle \nu\times\H^0,\bv\rangle_{\p\O, 1/2}
=\int_\O \langle \mP_n[\f(x,\Z+\h_2)]+\h_1,\bv\rangle dx.
$$

{\rm (ii)} We say $\Y$ is a weak solution of \eqref{MSNaN-f1Y} if $\Y\in L^{2,-1/2}_t(\O,\Bbb R^3)$ such that
$$
\mB(x,\Y(x))\in \mH(\O,\div0),\q \f(x,\mB(x,\Y(x)))\in L^{2,-1/2}_\nu(\O,\Bbb R^3),
$$
and for all $\bv\in H^1(\O,\Bbb R^3)$ it holds that
$$\aligned
\int_\O\langle \Y, \curl\bv\rangle dx+\langle \nu\times\H^0,\bv\rangle_{\p\O, 1/2}
=\int_\O \langle \mP_n[\f(x,\mB(x,\Y))]+\h_1,\bv\rangle dx.
\endaligned
$$
\end{Def}

Similar to Lemma \ref{Lem-MSnCN-f1Y} we have the following

\begin{Lem}\label{Lem-MSNaN-f1Y}  Assume \eqref{cond-MSf} and \eqref{cond-MSNaN-f1}.
 \begin{itemize}
 \item[(i)] If $\Y\in L^{2,-1/2}_t(\O,\Bbb R^3)$ is a weak solution of \eqref{MSNaN-f1Y}, then $\Y\in \mH(\O,\curl)$, the first equality in \eqref{MSNaN-f1Y} holds a.e. in $\O$, and the equality $\nu\times\Y=\nu\times\H^0$ holds in the sense of trace in $H^{-1/2}(\p\O,\Bbb R^3)$.
\item[(ii)] If $(\u,p)\in \mH(\O,\curl,\div0)\times H^1(\O)$ is a weak solution of \eqref{MSNaN-f1} and letting $\Z=\curl\u$, then $\Z\in \mH(\O,\div0)$ and it is a weak solution of \eqref{MSNaN-f1Z}.
On the other hand, if $\Z\in \mH(\O,\div0)$ is a weak solution of \eqref{MSNaN-f1Z}, then there exists $\h_2^0\in \Bbb H_2(\O)$ such that \eqref{MSNaN-f1} with $\h_2$ replaced by $\h_2+\h_2^0$ has a weak solution $(\u,p)\in H^1_{n0}(\O,\div0)\times H^1(\O)$, and $\curl\u=\Z-\h_2^0$. If furthermore $\Z\perp_{L^2(\O)}\Bbb H_2(\O)$, then $\h^0_2=\0$.

\item[(iii)] If $\Z\in \mH(\O,\div0)$ is a weak solution of \eqref{MSNaN-f1Z} and letting
$\Y=\mH(x,\Z(x)+\h_2(x)),$
then $\Y\in \mH(\O,\curl)$ and it is a weak solution of \eqref{MSNaN-f1Y}.
On the other hand, if $\Y\in L^{2,-1/2}_t(\O,\Bbb R^3)$ is a weak solution of \eqref{MSNaN-f1Y} and letting
$\Z=\mB(x,\Y(x))-\h_2(x),$
then $\Z\in \mH(\O,\div0)$ and it is a weak solution of \eqref{MSNaN-f1Z}.
\end{itemize}
\end{Lem}

\begin{Thm}\label{Thm-MANaN-f1} Assume \eqref{cond-MSf}, \eqref{cond-MSNaN-f1} and \eqref{f0}. Then for any given $\h_2\in\Bbb H_2(\O)$, there exists $\h_1\in\Bbb H_1(\O)$ which depends on $\h_2$ such that \eqref{MSNaN-f1} has a weak solution $(\u,p)\in H^1(\O,\Bbb R^3)\times H^1(\O)$.
\end{Thm}

\begin{proof} {\it Step 1}. To solve \eqref{MSNaN-f1Y}, we use \eqref{dec-2} to write
$\Y=\w+\hat\h_2+\nabla\psi$, with $\w\in \mH^\Gamma(\O,\div0)$, $\hat\h_2\in\Bbb H_2(\O)$ and $\psi\in H^1_0(\O)$.
If $\Y$ satisfies \eqref{MSNaN-f1Y} then $\psi$ satisfies
\eq\label{MSNaN-f1B}
\div\mB(x,\w+\hat\h_2+\nabla\psi)=0\q \text{in }\O,\q
\psi=0\q\text{on }\p\O.
\eeq
As in the proof of Proposition \ref{Prop-ex-MnC-BY} (also see Proposition \ref{Prop-sol} (i)), we can show that the map $\mT: H^1_0(\O)\to H^{-1}(\O)$ defined by
$$
\langle T[\psi],\eta\rangle_{H^{-1}(\O), H^1_0(\O)}=\int_\O \langle\mB(x,\w(x)+\hat\h_1(x)+\nabla\psi), \nabla\eta\rangle dx,\q\forall \eta\in H^1_0(\O).
$$
is hemi-continuous and strongly monotone. Then there exists a unique $\psi=\psi_{\w,\hat\h_2}\in H^1_0(\O)$ such that $\mT[\psi_{\w,\hat\h_2}]=0$.
The map $(\w,\hat\h_2)\mapsto \nabla\psi_{\w,\hat\h_2}$ is a locally Lipschitz and bounded map from $\mH^\Gamma(\O,\div0)\times\Bbb H_2(\O)$ into $\text{grad} H^1_0(\O)$.

{\it Step 2}. Denote
$$
\Y_{\w,\hat\h_2}=\w+\hat\h_2+\nabla\psi_{\w,\hat\h_2},\q
\Z_{\w,\hat\h_2}= \mB(x,\w+\hat\h_2+\nabla\psi_{\w,\hat\h_2})-\h_2.
$$
We look for $\bold v\in \mH^\Gamma(\O,\div0)$  such that
\eq\label{MSNaN-f1v}
\left\{\aligned
&\curl \bv=\mP_n[\f(x,\mB(x,\Y_{\w,\hat\h_2}))]+\h_1\q &\text{in }\O,\\
&\div\bv=0\q&\text{in }\O,\\
&\nu\times\bv=\nu\times\H^0\q&\text{on }\p\O.
\endaligned\right.
\eeq
We shall show that, there exists a unique $\h_1=\h_{1,\w,\hat\h_2}\in \Bbb H_1(\O)$ such that \eqref{MSNaN-f1v} has a unique solution $\bv=\bold V_{\hat\h_2}[\w]\in H^1(\O,\Bbb R^3)\cap \mH^\Gamma(\O,\div0)$, and
\eq\label{est-MSNaN}
\|\bold V_{\hat\h_2}[\w]\|_{H^1(\O)}\leq C\{\|\w+\hat\h_2\|_{L^2(\O)}+\|\nu\times\H^0\|_{H^{1/2}(\p\O)}\}+C,
\eeq
where $C$ depends on $\O, \mH(x,\z)$ and $\f(x,\z)$.

To prove,  we take a tangential component preserving extension of $\H^0_T$, which is also denoted by $\H^0$, see \cite{P2}. Then $\H^0\in H^1(\O,\div0)$. We can choose $\H^0$ such that
$$
\|\H^0\|_{H^1(\O)}\leq C(\O)\|\H^0_T\|_{H^{1/2}(\p\O)}\leq C(\O)\|\nu\times\H^0\|_{H^{1/2}(\p\O)}.
$$
Let $\z=\bv-\H^0$.
Then \eqref{MSNaN-f1v} is reduced to
\eq\label{MSNaN-f1vz}
\left\{\aligned
&\curl \z=\mP_n[\f(x,\mB(x,\Y_{\w,\hat\h_2}))]-\curl\H^0+\h_1\q &\text{in }\O,\\
&\div\z=0\q&\text{in }\O,\\
&\nu\times\z=\0\q&\text{on }\p\O.
\endaligned\right.
\eeq
Let us denote by $\bold g$  the right side of  the first equation in \eqref{MSNaN-f1vz}. Then \eqref{MSNaN-f1vz} is solvable in $H^1_{t0}(\O,\div0)$ if and only if $\bold g\in \mH^\Sigma_0(\O,\div0)$.

By the definition of $\mP_n$ we easily see that  $\bold g\in \mH_0(\O,\div0)$.
On the other hand, there exists a unique $\h_1=\h_{1,\w,\hat\h_2}$ such that $\bold g\perp \Bbb H_1(\O)$, namely
$$
\int_\O\langle \mP_n[\f(x,\mB(x,\Y_{\w,\hat\h_2}))]-\curl\H^0+\h_{1,\w,\hat\h_2},\h\rangle dx=0,\q\forall \h\in \Bbb H_1(\O).
$$
Recalling that $\mH_0(\O,\div0)=\mH^\Sigma_0(\O,\div0)\oplus_{L^2(\O)}\Bbb H_1(\O),$
we see that for this choice of $\h_1$ we have $\bold g\in \mH^\Sigma_0(\O,\div0)$. Then \eqref{MSNaN-f1vz} with $\h_1=\h_{1,\w,\hat\h_2}$ has a unique solution $\z\in H^1(\O,\div0)\cap \Bbb H_2(\O)^\perp_{L^2(\O)}= H_1(\O,\Bbb R^3)\cap \mH^\Gamma(\O,\div0)$.

To estimate $\h_{1,\w,\hat\h_2}$, let $\{\e_1,\cdots,\e_N\}$ be a orthonormal basis of $\Bbb H_1(\O)$. Then
$$\aligned
&\h_{1,\w,\hat\h_2}=\sum_{j=1}^N c_j(\w,\hat\h_2)\e_j,\\
&c_j(\w,\hat\h_2)=- \int_\O\langle \mP_n[\f(x,\mB(x,\Y_{\w,\hat\h_2}))]-\curl\H^0,\e_j\rangle dx.
\endaligned
$$
Using the above equalities and \eqref{H1H2} we have
$$\aligned
\|\h_{1,\w,\hat\h_2}\|_{H^1(\O)}\leq & C(\O)\|\h_{1,\w,\hat\h_2}\|_{L^2(\O)}\\
\leq& C(\O)\|\mP_n[\f(x,\mB(x,\Y_{\w,\hat\h_2}))]-\curl\H^0\|_{L^2(\O)}\\
\leq &C\{\|\Y_{\w,\hat\h_2}\|_{L^2(\O)}+\|\nu\times\H^0\|_{H^{1/2}(\p\O)}\}+C\\
\leq &C\{\|\w+\hat\h_2\|_{L^2(\O)}+\|\nu\times\H^0\|_{H^{1/2}(\p\O)}\}+C,
\endaligned
$$
where $C$ depends on $\O$, $\mH(x,\z)$ and $\f(x,\z)$.

For the solution $\z$ of \eqref{MSNaN-f1vz}  obtained above we have
$$\aligned
\|\z\|_{H^1(\O)}\leq& C(\O)\|\mP_n[\f(x,\mB(x,\Y_{\w,\hat\h_2}))]-\curl\H^0+\h_{1,\w,\hat\h_2}\|_{L^2(\O)}\\
\leq &C\{\|\w+\hat\h_2\|_{L^2(\O)}+\|\nu\times\H^0\|_{H^{1/2}(\p\O)}\}+C.
\endaligned
$$
Hence \eqref{MSNaN-f1v} has a unique solution $\bv=\z+\H^0$ in $\mH^\Gamma(\O,\div0)$. We denote this solution by $\bv=\bold V_{\hat\h_2}[\w]$.
Then we have \eqref{est-MSNaN}.

\eqref{est-MSNaN} shows that the map $\w\mapsto \bold V_{\hat\h_2}[\w]$ is a bounded map from $\mH(\O,\div0)$ into $H^1(\O,\Bbb R^3)\cap \mH^\Gamma(\O,\div0)$. Using $(H_1), (H_2), (H_3)$ and $(f_1)$ we see that the maps
$$
\w\mapsto \Y_{\w,\hat\h_2}\mapsto \mB(x,\Y_{\w,\hat\h_2})\mapsto \f(x,\mB(x,\Y_{\w,\hat\h_2}))\mapsto \mP_n[\f(x,\mB(x,\Y_{\w,\hat\h_2}))]
$$
are continuous in the $L^2$ norm. Thus the map
$\w\mapsto \h_{1,\w,\hat\h_2}$
is continuous in the $L^2$ norm. By these facts and using the $L^2$-div-curl-gradient inequalities we see that the map
$\w\mapsto \bold V_{\hat\h_2}[\w]$
is continuous from $\mH^\Gamma(\O,\div0)\to H^1(\O,\Bbb R^3)\cap\mH^\Gamma(\O,\div0)$.
By the Sobolev imbedding theorem we know that $\bold V_{\hat\h_2}$ is compact from $\mH^\Gamma(\O,\div0)$ into itself.

By the assumption \eqref{f0}, then there exists $R>0$ such that $\bold V_{\hat\h_2}$ maps the closed ball $B_R$ with center $\0$ and  radius $R$ in $\mH^\Gamma(\O,\div0)$ into $B_R$ itself. By Schauder's fixed point theorem we conclude that $\bold V_{\hat\h_2}$ has a fixed point $\w_{\hat\h_2}$ in $B_R$.

{\it Step 3}.
Let
$$\aligned
&\psi_{\hat\h_2}=\psi_{\w_{\hat\h_2},\hat\h_2},\q \h_{1,\hat\h_2}=\h_{1,\w_{\hat\h_2},\hat\h_2},\\
&\Y_{\hat\h_2}=\w_{\hat\h_2}+\hat\h_2+\nabla\psi_{\hat\h_2},\q
\Z_{\hat\h_2}= \mB(x,\w_{\hat\h_2}+\hat\h_2+\nabla\psi_{\hat\h_2})-\h_2.
\endaligned
$$
Since $\w_{\hat\h_2}$ satisfies
\eq\label{MSNaN-f1wh2}
\left\{\aligned
&\curl \w_{\hat\h_2}=\mP_n[\f(x,\mB(x,\Y_{\hat\h_2}))]+\h_{1,\hat\h_2}\q &\text{in }\O,\\
&\div\w_{\hat\h_2}=0\q&\text{in }\O,\\
&\nu\times\w_{\hat\h_2}=\nu\times\H^0\q&\text{on }\p\O,
\endaligned\right.
\eeq
we see that $\Y_{\hat\h_2}$ is a solution of \eqref{MSNaN-f1Y}, and by Lemma \ref{Lem-MSNaN-f1Y} (iii) $\Z_{\hat\h_2}$ is a solution of \eqref{MSNaN-f1Z}.

Using the method in the proof of Proposition \ref{Prop-sol} (ii), we can show that, for the given $\h_2\in\Bbb H_2(\O)$, there exists $\hat\h_2\in \Bbb H_2(\O)$ such that
\eq\label{MSNaN-f1-orth}
\int_\O\langle \mB(x,\w_{\hat\h_2}+\hat\h_2+\nabla\psi_{\hat\h_2})-\h_2, \h\rangle dx=0,\q\forall \h\in \Bbb H_2(\O).
\eeq
Namely, there exists $\hat\h_2$ such that $\Z_{\hat\h_2}\perp_{L^2(\O)}\Bbb H_2(\O)$.
Then $\Z_{\hat\h_2}\in \mH^\Gamma(\O,\div0)$.  By Lemma \ref{Lem-MSNaN-f1Y} (ii) we know that there exists $(\u,p)\in H_{n0}^1(\O,\div0)\times H^1(\O)$ which solves \eqref{MSNaN-f1} for the given $\h_2$.
\end{proof}

\subsection{The Co-normal-Neumann BVP}\

In this section we consider the natural BVP of the Maxwell-Stokes system
\eq\label{MSCoN-f1}
\left\{\aligned
&\curl [\mH(x,\curl\u+\h_2(x))]=\f(x,\curl\u+\h_2)+\h_1(x)+\nabla p\q &\text{in }\O,\\
&\div\u=0\q &\text{in }\O,\\
&\nu\cdot\mH(x,\curl\u+\h_2(x))=H^0_n,\q {\p p\over\p\nu}=-\nu\cdot\f(x,\curl\u+\h_2(x))\q&\text{on }\p\O.
\endaligned\right.
\eeq
In the following we assume \eqref{cond-MSf} and
$H_n^0\in H^{-1/2}(\p\O).$

\begin{Def}\label{Def-MSCoN-f1} We say that $(\u,p)$ is a weak solution of \eqref{MSCoN-f1} if $\u\in \mH(\O,\curl,\div0)$ and $p\in H^1(\O)$ such that
\begin{itemize}
\item[(i)] $\mH(x,\curl\u(x)+\h_2(x))\in L^{2,-1/2}_\nu(\O,\Bbb R^3)$, $\f(x,\curl\u(x)+\h_2(x))\in L^{2,-1/2}_\nu(\O,\Bbb R^3)$;
\item[(ii)] the equality $\nu\cdot\mH(x,\curl\u(x)+\h_2(x))=H^0_n$ holds in the $H^{-1/2}(\p\O)$ sense;
\item[(iii)]
$$\aligned
&\int_\O\langle\mH(x,\curl\u+\h_2),\curl\bv\rangle dx\\
=&\int_\O\langle \f(x,\curl\u+\h_2)+\h_1+\nabla p,\bv\rangle dx, \q \forall \bv\in H^1_{t0}(\O,\Bbb R^3),
\endaligned
$$
$$
\int_\O\langle \f(x,\curl\u+\h_2)+\nabla p,\nabla\eta\rangle dx=0, \q \forall \eta\in H^1(\O).
$$
\end{itemize}
\end{Def}

We shall show the relation between \eqref{MSCoN-f1} with the $Z$-system
\eq\label{MSCoN-f1Z}
\left\{\aligned
&\curl [\mH(x,\Z+\h_2)]=\mP_\nu[\f(x,\Z+\h_2)]+\h_1\q &\text{in }\O,\\
&\div\Z=0\q &\text{in }\O,\\
&\nu\cdot\mH(x,\Z+\h_2)=H^0_n\q& \text{on }\p\O,
\endaligned\right.
\eeq
and we shall show the equivalence of $Z$-system \eqref{MSCoN-f1Z} and the $Y$-system
\eq\label{MSCoN-f1Y}
\left\{\aligned
&\curl\Y=\mP_\nu[\f(x,\mB(x,\Y))]+\h_1\q &\text{in }\O,\\
&\div\mB(x,\Y)=0\q &\text{in }\O,\\
&\nu\cdot\Y=H^0_n\q& \text{on }\p\O.
\endaligned\right.
\eeq

\begin{Def}\label{Def-MSCoN-f1Z} {\rm (i)}
We say $\Z$ is a weak solution of \eqref{MSCoN-f1Z} if $\Z\in \mH(\O,\div0)$ such that
$$\aligned
&\mH(x,\Z(x)+\h_2(x))\in L^{2,-1/2}_\nu(\O,\Bbb R^3),\q \f(x,\Z(x)+\h_2(x))\in L^2(\O,\Bbb R^3),\\
&\nu\cdot\mH(x,\Z(x)+\h_2(x))=H^0_n\q\text{in the sense of $H^{-1/2}(\p\O)$},\\
&\int_\O\langle \mH(z,\Z+\h_2), \curl\bv\rangle dx=\int_\O \langle \mP_\nu[\f(x,\Z+\h_2)]+\h_1,\bv\rangle dx,\q\forall \bv\in H^1_{t0}(\O,\Bbb R^3).
\endaligned
$$

{\rm (ii)} We say $\Y$ is a weak solution of \eqref{MSNaN-f1Y} if $\Y\in L^{2,-1/2}_\nu(\O,\Bbb R^3)$ such that
$$\aligned
&\mB(x,\Y(x))\in \mH(\O,\div0),\q \f(x,\mB(x,\Y(x)))\in L^2(\O,\Bbb R^3),\\
&\nu\cdot\Y=H^0_n\q\text{holds in the sense of $H^{-1/2}(\p\O)$,}\\
&\int_\O\langle \Y, \curl\bv\rangle dx
=\int_\O \langle \mP_\nu[\f(x,\mB(x,\Y))]+\h_1,\bv\rangle dx,\q\forall \bv\in H^1_{t0}(\O,\Bbb R^3).
\endaligned
$$
\end{Def}

Similar to Lemma \ref{Lem-MSnCN-f1Y} we have

\begin{Lem}\label{Lem-MSCoN-f1Y}  Assume \eqref{cond-MSf} and $H^0_n\in H^{-1/2}(\p\O)$.
\begin{itemize}
\item[(i)] If $\Y\in L^{2,-1/2}_\nu(\O,\Bbb R^3)$ is a weak solution of \eqref{MSCoN-f1Y}, then $\Y\in \mH(\O,\curl)$, and the first equality in \eqref{MSCoN-f1Y} holds a.e. in $\O$.

\item[(ii)] If $(\u,p)\in \mH(\O,\curl,\div0)\times H^1(\O)$ is a weak solution of \eqref{MSCoN-f1} and letting $\Z=\curl\u$, then $\Z\in \mH(\O,\div0)$ and it is a weak solution of \eqref{MSCoN-f1Z}.
On the other hand, if $\Z\in \mH(\O,\div0)$ is a weak solution of \eqref{MSCoN-f1Z}, then there exists $\h_2^0\in \Bbb H_2(\O)$ such that \eqref{MSCoN-f1} with $\h_2$ replaced by $\h_2+\h_2^0$ has a weak solution $(\u,p)\in H^1_{n0}(\O,\div0)\times H^1(\O)$ with $\curl\u=\Z-\h_2^0$. If furthermore $\Z\perp_{L^2(\O)}\Bbb H_2(\O)$, then $\h^0_2=\0$.

\item[(iii)] If $\Z\in \mH(\O,\div0)$ is a weak solution of \eqref{MSCoN-f1Z} and letting
$\Y=\mH(x,\Z(x)+\h_2(x)),$
then $\Y\in \mH(\O,\curl)$ and it is a weak solution of \eqref{MSCoN-f1Y}.
On the other hand, if $\Y\in L^{2,-1/2}_\nu(\O,\Bbb R^3)$ is a weak solution of \eqref{MSCoN-f1Y} and letting
$\Z=\mB(x,\Y(x))-\h_2(x),$
then $\Z\in \mH(\O,\div0)$ and it is a weak solution of \eqref{MSCoN-f1Z}.
\end{itemize}
\end{Lem}

\begin{Thm}  Assume \eqref{cond-MSf}, $(H_{3N})$, $(H^0)$ and \eqref{f0}. Then for any given $\h_1, \hat\h_1\in\Bbb H_1(\O)$ and $\h_2=\h_2(\hat\h_1)\in\Bbb H_2(\O)$ which depends on $\hat\h_1$ such that \eqref{MSCoN-f1} has a weak solution $(\u,p)\in H^1(\O,\Bbb R^3)\times H^1(\O)$ for the given $\h_1$ and for $\h_2=\h_2(\hat\h_1)$.
\end{Thm}

\begin{proof} {\it Step 1}. By Lemma \ref{Lem-MSCoN-f1Y}, we first solve  \eqref{MSCoN-f1Y}. Using \eqref{dec-1}, we write
$$\Y=\w+\hat\h_1+\nabla\psi,\q \w\in \mH^\Sigma_0(\O,\div0),\q \hat\h_1\in\Bbb H_1(\O),\q \psi\in H^1(\O).
$$
From the last two equalities in \eqref{MSCoN-f1Y} and using $(H^0)$ we have
$$
\left\{\aligned
&\div[\mB(x,\w+\hat\h_1+\nabla\psi)]=0\q &\text{in }\O,\\
&\nu\cdot(\w+\hat\h_1+\nabla\psi)=\nu\cdot\H^0\q&\text{on }\p\O.
\endaligned\right.
$$
Using $(H_3)$ and $(H_{3N})$ we see that this BVP can be written as
\eq\label{MSCoN-fB}
\left\{\aligned
&\div[\mB(x,\w+\hat\h_1+\nabla\psi)]=0\q &\text{in }\O,\\
&\nu\cdot\mB(x,\w+\hat\h_1+\nabla\psi)=\nu\cdot\mB(x,\H^0)\q&\text{on }\p\O.
\endaligned\right.
\eeq
Hence $(\w,\hat\h_1,\nabla\psi)$ satisfies \eqref{MSCoN-fB}.

By the monotone operator methods we can show that, for any given $\w\in \mH^\Sigma_0(\O,\div0)$ and $\hat\h_1\in\Bbb H_1(\O)$,  \eqref{MSCoN-fB} has a unique solution $\psi=\psi_{\w,\hat\h_1}\in \dot H^1(\O)$.

{\it Step 2}. Denote
$$
\Y_{\w,\hat\h_1}=\w+\hat\h_1+\nabla\psi_{\w,\hat\h_1},\q
\Z_{\w,\hat\h_1}= \mB(x,\w+\hat\h_1+\nabla\psi_{\w,\hat\h_1})-\h_2.
$$
We look for $\bold v\in \mH^\Sigma_0(\O,\div0)$  such that
\eq\label{MSCoN-f1v}
\left\{\aligned
&\curl \bv=\mP_\nu[\f(x,\mB(x,\Y_{\w,\hat\h_1}))]+\h_1\q &\text{in }\O,\\
&\div\bv=0\q&\text{in }\O,\\
&\nu\cdot\bv=0\q&\text{on }\p\O.
\endaligned\right.
\eeq
By the definition of $\mP_\nu$ we see that the right side of the first equation of \eqref{MSCoN-f1v} lies in $\mH^\Gamma(\O,\div0)$.
So \eqref{MSCoN-f1v} has a unique solution $\bv=\bold V_{\hat\h_1}[\w]\in H^1_{n0}(\O,\div0)\cap \Bbb H_1(\O)^\perp_{L^2(\O)}= H_1(\O,\Bbb R^3)\cap \mH^\Sigma_0(\O,\div0)$. Similar to Lemma \ref{Lem-MSnCN-f1v} we can show that
$$
\|\bold V_{\hat\h_1}[\w]\|_{H^1(\O)}\leq C\{\|\w+\hat\h_1\|_{L^2(\O)}+\|H^0_n\|_{H^{-1/2}(\p\O)}\}+C,
$$
where $C$ depends on $\O, \mH(x,\z)$ and $\f(x,\z)$.
Using the Sobolev imbedding theorem we know that $\bold V_{\hat\h_1}$ is compact from $\mH^\Sigma_0(\O,\div0)$ into itself.

Since $\f(x,\z)$ satisfies \eqref{f0}, then there exists $R>0$ such that $\bold V_{\hat\h_1}$ maps the closed ball $B_R$ with center $\0$ and  radius $R$ in $\mH^\Sigma_0(\O,\div0)$ into $B_R$ itself. By Schauder's fixed point theorem we conclude that $\bold V_{\hat\h_1}$ has a fixed point $\w_{\hat\h_1}$ in $B_R$.

{\it Step 3}.
Let
$$\aligned
&\psi_{\hat\h_1}=\psi_{\w_{\hat\h_1},\hat\h_1}, \q   \Y_{\hat\h_1}=\w_{\hat\h_1}+\hat\h_1+\nabla\psi_{\hat\h_1},\\
&\Z_{\hat\h_1,\h_2}= \mB(x,\w_{\hat\h_1}+\hat\h_1+\nabla\psi_{\hat\h_1})-\h_2.
\endaligned
$$
$\w_{\hat\h_1}$ satisfies
\eq\label{MSCoN-f1wh2}
\left\{\aligned
&\curl \w_{\hat\h_1}=\mP_\nu[\f(x,\mB(x,\Y_{\hat\h_1}))]+\h_1\q &\text{in }\O,\\
&\div\w_{\hat\h_1}=0\q&\text{in }\O,\\
&\nu\cdot\w_{\hat\h_1}=H^0_n\q&\text{on }\p\O.
\endaligned\right.
\eeq
Then $\Y_{\hat\h_1}$ is a solution of \eqref{MSCoN-f1Y}, and by Lemma \ref{Lem-MSCoN-f1Y} (iii) we know that $\Z_{\hat\h_1,\h_2}$ is a solution of \eqref{MSCoN-f1Z}.

Now we choose $\h_2=\h_2(\hat\h_1)$ such that
\eq\label{MSCoN-f1-orth}
\int_\O\langle \mB(x,\w_{\hat\h_1}+\hat\h_1+\nabla\psi_{\hat\h_1})-\h_2(\hat\h_1), \h\rangle dx=0,\q\forall \h\in \Bbb H_2(\O).
\eeq
Denote $\Z_{\hat\h_1}=\Z_{\hat\h_1,\h_2(\hat\h_1)}$. Then $\Z_{\hat\h_1}\in \mH^\Gamma(\O,\div0)$.  By Lemma \ref{Lem-MSCoN-f1Y} (ii), there exists $(\u,p)\in H_{n0}^1(\O,\div0)\times H^1(\O)$ which solves \eqref{MSCoN-f1} for the given $\h_1$ and for $\h_2=\h_2(\hat\h_1)$.

\end{proof}

\subsection*{Acknowledgements}  This work was partially supported
by the National Natural Science Foundation of China grants no.  11671143 and no. 11431005.

\enddocument